\documentclass{amsart}
\usepackage[framemethod=tikz]{mdframed}

\usepackage[most]{tcolorbox}

\usepackage{circuitikz}
\usetikzlibrary{arrows,matrix,positioning}
\usepackage{hyperref}
\usepackage{lipsum}
\usepackage{amsfonts}
\usepackage{graphicx}
\usepackage{epstopdf}
\usepackage{algorithmic}
\ifpdf
  \DeclareGraphicsExtensions{.eps,.pdf,.png,.jpg}
\else
  \DeclareGraphicsExtensions{.eps}
\fi

\usepackage{amsopn}

\newcommand{\F}{\mathcal{F}}
\newtheorem{theorem}{Theorem}[section]

\newtheorem{conjecture}[theorem]{Conjecture}
\newtheorem{definition}[theorem]{Definition}

\newtheorem{lemma}[theorem]{Lemma}

\newtheorem{question}[theorem]{Question}
\newtheorem{proposition}[theorem]{Proposition} 

\theoremstyle{definition}
\newtheorem{algorithm}[theorem]{Algorithm}
\theoremstyle{remark}
\newmdtheoremenv[
hidealllines=true,
leftline=true,
innertopmargin=0pt,
innerbottommargin=0pt,
linewidth=4pt,
linecolor=gray!40,
innerrightmargin=0pt,
innertopmargin=0pt,
]{examplei}{Example}

\begin{document}

\title[Algorithmic techniques for finding resistance distance]{Algorithmic techniques for finding resistance distances on structured graphs}

\author[E. J. Evans]{E. J. Evans}

\address{Department of Mathematics,
  Brigham Young University,
  Provo, Utah 84602, USA}
\author[A. E. Francis]{A. E. Francis}

\address{Mathematical Reviews, American Mathematical Society, 
  Ann Arbor, Michigan 48103, USA}
\thanks{This material is based upon work supported by the National Science Foundation under Grant No. 1440140, while the authors were in residence at the Mathematical Sciences Research Institute in Berkeley, California, during the summer of 2019}

\begin{abstract}
  In this paper we give a survey of methods used to calculate values of resistance distance (also known as effective resistance) in graphs. Resistance distance has played a prominent role not only in circuit theory and chemistry, but also in combinatorial matrix theory and spectral graph theory. Moreover resistance distance has applications ranging from quantifying biological structures, distributed control systems, network analysis, and power grid systems.  In this paper we discuss both exact techniques and approximate techniques and for each method discussed we provide an illustrative example of the technique. We also present some open questions and conjectures. 
\end{abstract}


\maketitle

\section{Introduction}
The resistance distance (occasionally referred to as the effective resistance) of a graph is a measure that quantifies its structural properties.  Resistance distance has its origin in electrical circuit theory and its first known application to graph structure occurred in the analysis of chemical structure~\cite{KleinRandic}.  Resistance distance in graphs has played a prominent role not only in circuit theory and chemistry~\cite{doylesnell,KleinRandic,oldbook}, but also in combinatorial matrix theory~\cite{Bapatbook,YangKlein} and spectral graph theory~\cite{bapatdvi,mgt,chenzhang,SpielSparse}.

A few specific examples of the use of resistance distance are:
\begin{itemize}
    \item Spielman and Srivastava~\cite{SpielSparse} have used resistance distance between nodes of graphs to develop an algorithm to rapidly sparsify a given graph while maintaining spectral properties.
    \item Ghosh, Boyd, and Saberi~\cite{Ghosh} considered the problem of minimizing the total resistance distance by allocating edge weights on a given graph.  This problem has applications to Markov chains and continuous-time averaging networks. 
    \item Resistance distance gives sharp upper and lower bounds for Kemeny's constant~\cite{kem1}, an important constant in the theory of random walks.  
    \item Effective resistance is used in the field of distributed control and estimation.  In particular, one problem in this field is the estimation of several variables in the presence of noisy data. This is of particular interest in the design and operation of sensor arrays~\cite{Barooah06grapheffective}.
    \item Both the Wiener index and the Balaban index can be determined from the resistance distances in the graph~\cite{rdmatrix}. The Wiener index is a topological index of a molecule; the Balaban index $J(G)$ is defined by 
    \[
    J(G)=\frac{m}{m-n+2}\sum_{\text{edges}\in G}\frac{1}{\sqrt{w(u)\cdot w(v)}},
    \] 
    where $n=|V|$ and $m=|E|$ and $w(u)$ denotes the sum of distances from $u$  to all the other vertices of $G$. 
    \item Resistance distance has been used extensively in topological analysis of different chemical compounds and structures.  A very short list of sample papers include~\cite{carmona2014effective,KleinRandic, klein2002resistance, peng2017kirchhoff,wang2010kirchhoff,yang2014comparison, yang2008kirchhoff}.
\end{itemize}

The goal of this paper is to survey common methods of determining the resistance distance in graphs and to provide examples of these methods (for additional reading see in~\cite{vos2016methods} and~\cite{YangKlein}). 

The structure of the paper is as follows.  In Section~\ref{sec:def} we give a formal definition of resistance distance, and describe circuit transformations that can be applied to electric circuits that are useful in determining resistance distance between vertices of a graph.  In Section~\ref{sec:tools} we introduce mathematical  techniques used to determine resistance distances in graphs, that do not require the use of circuits.  For each of the techniques which give an exact answer, we provide a representative example.  We conclude this section by briefly reviewing numerical techniques that give reasonable estimates for resistance distance and are useful in the case of very large graphs. 
We conclude with a collection of open conjectures and questions.

\section{Resistance Distance and Circuit Transformations}\label{sec:def}
With it's origin in circuit theory, it is not surprising that the formal definition of resistance distance is stated in the language of this discipline.
\begin{definition}

Given a graph $G$ we assume that the graph $G$ represents an electrical circuit with resistances on each edge.  The resistance on a weighted edge is the reciprocal of its edge weight. Given any two nodes $i$ and $j$ assume that one unit of current flows into node $i$ and one unit of current flows out of node $j$.  The potential difference $v_i - v_j$ between nodes $i$ and $j$ needed to maintain this current is the \emph{resistance distance} between $i$ and $j$.  By Ohm's law, the resistance, $r_G(i,j)$ is 
\[
r_G(i,j) = \frac{v_i-v_j}{1}=v_i-v_j.
\]
\end{definition}

 We note that unless otherwise stated we will assume that each edge has weight one (corresponding to a resistance of one Ohm).
\subsection{Electric Circuit Transformations}\label{sec:networktransformations}
One method for determining the effective resistance in a graph is to turn to the techniques of circuit analysis.  These techniques include the well-known series and parallel rules and the $\Delta$--Y and  Y--$\Delta$ transformations.

\begin{definition}[Series Transformation] Let $N_1$, $N_2$, and $N_3$ be nodes in a graph where $N_2$ is adjacent to only $N_1$ and $N_3$.  Moreover, let $R_A$ equal the resistance between $N_1$ and $N_2$ and $R_B$ equal the resistance between node $N_2$ and $N_3$.  Under a series transformation on the graph, $N_2$ is deleted and the resistance between $N_1$ and $N_3$ is set equal to $R_C = R_A + R_B$.\end{definition}

\begin{definition}[Parallel Transformation] Let $N_1$ and $N_2$ be nodes in a multi-edged graph where $e_1$ and $e_2$ are two edges between $N_1$ and $N_2$ with resistances $R_A$ and $R_B$, respectively.   Under a parallel transformation on the graph, the edges $e_1$ and $e_2$ are deleted and a new edge is added between $N_1$ and $N_2$ with edge resistance ${R_C = \left(\frac{1}{R_A} + \frac{1}{R_B}\right)^{-1}}$.
\end{definition}

Next, we recall $\Delta$--Y and Y--$\Delta$ transformations, which are mathematical techniques to convert between resistors in a triangle ($\Delta$) formation and an equivalent system of three resistors in a ``Y'' format as illustrated in Figure~\ref{fig:dy}.  We formalize these transformations below. 
\begin{definition}[$\Delta$--Y transformation]\label{def:dy}
Let $N_1, N_2, N_3$ be nodes and $R_A$, $R_B$ and $R_C$ be given resistances as shown in Figure~\ref{fig:dy}.  The transformed circuit in the ``Y'' format as shown in Figure~\ref{fig:dy} has the following resistances:
\begin{align*}
  R_1 &= \frac{R_BR_C}{R_A + R_B + R_C} \\
  R_2 &= \frac{R_AR_C}{R_A + R_B + R_C} \\
  R_3 &= \frac{R_AR_B}{R_A + R_B + R_C}
\end{align*}
\end{definition}
\begin{definition}[Y--$\Delta$ transformation]\label{def:yd}
Let $N_1, N_2, N_3$ be nodes and $R_1$, $R_2$ and $R_3$ be given resistances as shown in Figure~\ref{fig:dy}.  The transformed circuit in the ``$\Delta$'' format as shown in Figure~\ref{fig:dy} has the following resistances:
\begin{align*}
  R_A &= \frac{R_1R_2 + R_2R_3 + R_1R_3}{R_1} \\
  R_B &= \frac{R_1R_2 + R_2R_3 + R_1R_3}{R_2} \\
  R_C &= \frac{R_1R_2 + R_2R_3 + R_1R_3}{R_3}
\end{align*}
\end{definition}
 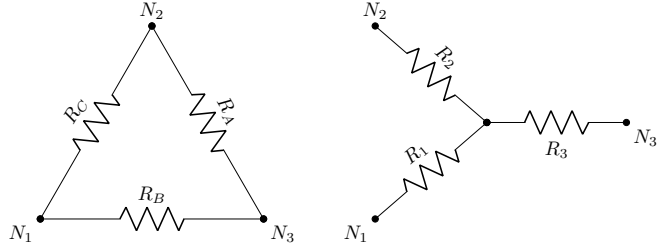
\begin{figure}
 \begin{center}
\resizebox{3.5in}{!}{ \begin{circuitikz}
 
 \draw (0,0) to [R, *-*,l=$R_B$] (4,0);
  \draw (0,0) to [R, *-*,l=$R_C$] (2,3.46);
  \draw (2,3.46) to[R, *-*,l=$R_A$] (4,0)
  {[anchor=north east] (0,0) node {$N_1$} } {[anchor=north west] (4,0) node {$N_3$}} {[anchor=south]  (2,3.46) node {$N_2$}};
  
  \draw (6,0) to [R, *-*,l=$R_1$] (8,1.73);
  \draw (6,3.46) to [R, *-*,l=$R_2$] (8,1.73);
 \draw (10.5,1.73) to [R, *-*,l=$R_3$] (8,1.73)
  {[anchor=north east] (6,0) node {$N_1$} } {[anchor=north west] (10.5,1.73) node {$N_3$}} {[anchor=south]  (6,3.46) node {$N_2$}};
 ; 
 \end{circuitikz}}
 \end{center}
 \caption{ $\Delta$ and $Y$ circuits with vertices labeled as in Definition~\ref{def:dy}.}
 \label{fig:dy}
 \end{figure}
\begin{proposition} Series, parallel, $\Delta$--Y, and Y--$\Delta$ transformations yield equivalent electric circuits.
\end{proposition}
\begin{proof}
See~\cite{oldbook} for a proof of this result.
\end{proof}

In addition to the network transformations just described, the cut-vertex theorem can also be used to calculate resistance distances.  
\begin{theorem}[Cut Vertex]~\cite{bent2tree}\label{thm:cutvertex}
Let $G$ be a connected graph with weights $w(i,j)$ for every edge of $G$ and suppose $v$ is a cut-vertex of $G$.  Let $C$ be a component of $G - v$ and let $H$ be the induced subgraph on $V(C) \cup \{v\}$. Then for each pair of vertices $i$, $j$ of $H$, 
\begin{equation}\label{eq:star}
r_G(i,j) = r_H(i,j).
\end{equation}
\end{theorem}

\begin{examplei}\label{ex:laddertransform} An interesting example of using equivalent network transformations and was provided by Cinkir \cite{Cinkir}. He used some clever recurrence relations and knowledge of the existence of equivalent network transformations to find formulae for the effective resistance between any pair of nodes in a ladder graph $L_n$ with $2n$ vertices, for arbitrary $n$ shown in Figure \ref{fig:laddergraph} (A). 

Here, we reproduce some of his results using a strategy involving explicit network transformations. We do this to demonstrate how these transformations can be used on families of structured graphs. 

We begin by performing a series transformation on $L_{n}$ (deleting node 2) to obtain the equivalent network, shown in the right panel of Figure \ref{fig:laddergraph}, and consider the following algorithmic process.

\begin{algorithm}\label{alg:ladder}
Let $G$ be a graph of the form shown in (A) of Figure \ref{fig:ladderalg}, where edge resistances are assumed to be equal to one unless labeled otherwise. \begin{enumerate}
    \item Perform a $\Delta$-$Y$ transformation on the (leftmost) triangle.  
    \item Perform series transformations to eliminate nodes $2i-1$ and $2i$. 
\end{enumerate} 
The obtained graph with its non-unit edge resistances is shown in (B) of Figure \ref{fig:ladderalg}. Note that this process exchanges a triangle for a new ``tail'' edge with resistance $t_i$ as in Figure~\ref{fig:ladderalg}.
It is straightforward to check that 
\begin{equation}\label{eq:ladderalg}
    a_i = \frac{2a_{i-1}+b_{i-1}+1}{a_{i-1}+b_{i-1}+1},\quad 
    b_i =\frac{a_{i-1}+2b_{i-1}+1}{a_{i-1}+b_{i-1}+1},  \quad 
    t_i = \frac{a_{i-1}b_{i-1}}{a_{i-1}+b_{i-1}+1}
\end{equation}
\end{algorithm}

Starting from the graph in (B) of Figure \ref{fig:laddergraph}, we have initial values $a_0 = 1$, $b_0 = 2$. After the first iteration of Algorithm \ref{alg:ladder}, we have $a_1 =\frac{5}{4}$, $b_1 =\frac{6}{4}$, and $t_1 = \frac{1}{2} $. The edge resistances $a_i, b_i,t_i$ after the first 5 iterations of this algorithm are shown in Table \ref{tab:laddergraph}.
The numerator $x_i$ of the $a_i$ is sequence A061278 in OEIS \cite{OEIS}, which satisfies the recurrence relation 
\[
x_i = 4x_{i-1}-x_{i-2}+1, 
\] and which can be determined to be 
\[ 
x_i = \frac{1}{12}\big((3 - \sqrt{3})(2 - \sqrt{3})^{i+1} + (3 + \sqrt{3})(2 + \sqrt{3})^{i+1}-6\big).
\]
Note that the numerator of $b_i$ is simply $x_i+1$ and that the denominator at step $i$ is equal to $x_i - x_{i-1}$.

After $n-1$ steps we have transformed $L_n$ into the equivalent network in Figure \ref{fig:ladderalgend} (A). We perform one final $\Delta$-$Y$ transformation to obtain the network in Figure \ref{fig:ladderalgend} (B). From this graph we can now use Theorem \ref{thm:cutvertex} and the symmetry of the original graph to determine that 

\begin{gather*}
r(1,2n-1) = r(2,2n) = \frac{n}{2} + \frac{a_{n-1}}{a_{n-1}+b_{n-1}+1}\\
r(1,2n) = r(2,2n-1) = \frac{n}{2} + \frac{b_{n-1}}{a_{n-1}+b_{n-1}+1}\\
r(1,2) = r(2n-1,2n) =  \frac{a_{n-1}+b_{n-1}}{a_{n-1}+b_{n-1}+1}\\
\end{gather*}

Resistances between other pairs of vertices in $L_n$ can be obtained in a similar fashion. 
\end{examplei}
\begin{table}[ht!]
    \centering
    \begin{tabular}{|c|c|c|c|}
    \hline
    $i$ & $a_i$ & $b_i$ & $t_i$ \\ \hline 
    1 & 5/4 & 6/4 & 1/2 \\
    2 & 20/15 & 21/15 & 1/2  \\
    3 & 76/56 & 77/56 & 1/2   \\ 
    4 &  285/209 & 286/209 & 1/2  \\
    5 & 1065 / 780 & 1066/780 & 1/2\\\hline
    \end{tabular}\\[3mm]
    \caption{The values of the resistances in the tails and triangle edges for the first five iterations of Algorithm \ref{alg:ladder} for the ladder graph.}
    \label{tab:laddergraph}
\end{table}
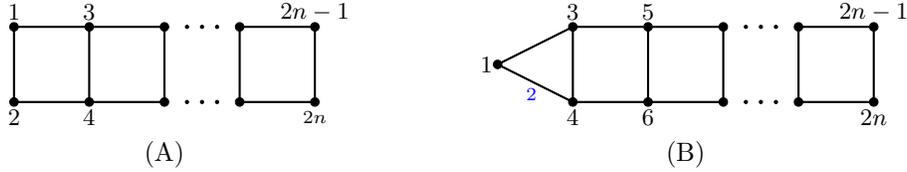
\begin{figure}[ht!]
\begin{center}
\begin{tikzpicture}[line cap=round,line join=round,>=triangle 45,x=1.0cm,y=1.0cm,scale = 1]
\draw [line width=.8pt] (3.,4.)-- (3.,3.);
\draw [line width=.8pt] (3.,4.)-- (4.,4.);
\draw [line width=.8pt] (3.,3.)-- (4.,3.);
\draw [line width=.8pt] (4.,4.)-- (4.,3.);
\draw [line width=.8pt] (4.,4.)-- (5.,4.);
\draw [line width=.8pt] (4.,3.)-- (5.,3.);
\draw [line width=.8pt] (5.,4.)-- (5.,3.);
\draw [fill=black] (5.3,4.) circle (.6pt);
\draw [fill=black] (5.5,4.) circle (.6pt);
\draw [fill=black] (5.7,4.) circle (.6pt);
\draw [fill=black] (5.3,3.) circle (.6pt);
\draw [fill=black] (5.5,3.) circle (.6pt);
\draw [fill=black] (5.7,3.) circle (.6pt);
\draw [line width=.8pt] (6.,4.)-- (6.,3.);
\draw [line width=.8pt] (6.,4.)-- (7.,4.);
\draw [line width=.8pt] (6.,3.)-- (7.,3.);
\draw [line width=.8pt] (7.,4.)-- (7.,3.);
\begin{small}
\draw [fill=black] (3.,4.) circle (1.6pt);
\draw[color=black] (3,4.2) node {$1$};
\draw [fill=black] (3.,3.) circle (1.6pt);
\draw[color=black] (3,2.8) node {$2$};
\draw [fill=black] (4.,4.) circle (1.6pt);
\draw[color=black] (4,4.2) node {$3$};
\draw [fill=black] (4.,3.) circle (1.6pt);
\draw[color=black] (4,2.8) node {$4$};
\draw [fill=black] (5.,4.) circle (1.6pt);
\draw [fill=black] (5.,3.) circle (1.6pt);
\draw [fill=black] (6.,4.) circle (1.6pt);
\draw [fill=black] (6.,3.) circle (1.6pt);
\draw [fill=black] (7.,4.) circle (1.6pt);
\draw[color=black] (7,4.2) node {$2n-1$};
\draw [fill=black] (7.,3.) circle (1.6pt);
\end{small}
\begin{scriptsize}

\draw[color=black] (7,2.8) node {$2n$};

\end{scriptsize}
\draw (5,2.3) node {(A)};
\end{tikzpicture}
\qquad\qquad
\begin{tikzpicture}[line cap=round,line join=round,>=triangle 45,x=1.0cm,y=1.0cm,scale = 1]
\draw [line width=.8pt] (2.,3.5)-- (3.,4.);
\draw [line width=.8pt] (2.,3.5)-- (3.,3.);
\draw [line width=.8pt] (3.,4.)-- (3.,3.);
\draw [line width=.8pt] (3.,4.)-- (4.,4.);
\draw [line width=.8pt] (3.,3.)-- (4.,3.);
\draw [line width=.8pt] (4.,4.)-- (4.,3.);
\draw [line width=.8pt] (4.,4.)-- (5.,4.);
\draw [line width=.8pt] (4.,3.)-- (5.,3.);
\draw [line width=.8pt] (5.,4.)-- (5.,3.);
\draw [fill=black] (5.3,4.) circle (.6pt);
\draw [fill=black] (5.5,4.) circle (.6pt);
\draw [fill=black] (5.7,4.) circle (.6pt);
\draw [fill=black] (5.3,3.) circle (.6pt);
\draw [fill=black] (5.5,3.) circle (.6pt);
\draw [fill=black] (5.7,3.) circle (.6pt);
\draw [line width=.8pt] (6.,4.)-- (6.,3.);
\draw [line width=.8pt] (6.,4.)-- (7.,4.);
\draw [line width=.8pt] (6.,3.)-- (7.,3.);
\draw [line width=.8pt] (7.,4.)-- (7.,3.);
\begin{small}
\draw [fill=black] (2.,3.5) circle (1.6pt);
\draw[color=black] (1.85,3.5) node {$1$};
\draw [fill=black] (3.,4.) circle (1.6pt);
\draw[color=black] (3,4.2) node {$3$};
\draw [fill=black] (3.,3.) circle (1.6pt);
\draw[color=black] (3,2.8) node {$4$};
\draw [fill=black] (4.,4.) circle (1.6pt);
\draw[color=black] (4,4.2) node {$5$};
\draw [fill=black] (4.,3.) circle (1.6pt);
\draw[color=black] (4,2.8) node {$6$};
\draw [fill=black] (5.,4.) circle (1.6pt);
\draw [fill=black] (5.,3.) circle (1.6pt);
\draw [fill=black] (6.,4.) circle (1.6pt);
\draw [fill=black] (6.,3.) circle (1.6pt);
\draw [fill=black] (7.,4.) circle (1.6pt);
\draw[color=black] (7,4.2) node {$2n-1$};
\draw [fill=black] (7.,3.) circle (1.6pt);
\draw[color=black] (7,2.8) node {$2n$};
\begin{scriptsize}
\draw[color=blue] (2.45,3.1) node {$2$};

\end{scriptsize}

\end{small}
\draw (4.5,2.3) node {(B)};
\end{tikzpicture}

\end{center}
\caption{Demonstration of a series transformation on a ladder graph.  The graph in the left panel is the original ladder graph, and the graph in the right panel is the graph after performing a series transformation to remove node 2.  We note that the edge $(1,4)$ in the right graph has resistance  equal to the resistance of the sum of the resistances of edges $(1,2)$ and $(2,4)$ in the left graph.}\label{fig:laddergraph}
\end{figure}
\begin{figure}[ht!]
\begin{center}
    \begin{tikzpicture}[line cap=round,line join=round,>=triangle 45,x=1.0cm,y=1.0cm,scale = 1]
\draw [fill=black] (1.3,3.5) circle (.6pt);
\draw [fill=black] (1.5,3.5) circle (.6pt);
\draw [fill=black] (1.7,3.5) circle (.6pt);

\draw [line width=.8pt] (2.,3.5)-- (3.,4.);
\draw [line width=.8pt] (2.,3.5)-- (3.,3.);
\draw [line width=.8pt] (3.,4.)-- (3.,3.);
\draw [line width=.8pt] (3.,4.)-- (4.,4.);
\draw [line width=.8pt] (3.,3.)-- (4.,3.);
\draw [line width=.8pt] (4.,4.)-- (4.,3.);
\draw [line width=.8pt] (4.,4.)-- (5.,4.);
\draw [line width=.8pt] (4.,3.)-- (5.,3.);
\draw [line width=.8pt] (5.,4.)-- (5.,3.);
\draw [fill=black] (5.3,4.) circle (.6pt);
\draw [fill=black] (5.5,4.) circle (.6pt);
\draw [fill=black] (5.7,4.) circle (.6pt);
\draw [fill=black] (5.3,3.) circle (.6pt);
\draw [fill=black] (5.5,3.) circle (.6pt);
\draw [fill=black] (5.7,3.) circle (.6pt);
\begin{small}
\draw [fill=black] (2.,3.5) circle (1.6pt);
\draw[color=black] (1.9,3.8) node {$V_0$};

\end{small}
\begin{scriptsize}
\draw[color=blue] (2.45,3.9) node {$a_0$};
\draw[color=blue] (2.45,3.1) node {$b_0$};

\draw [fill=black] (3.,4.) circle (1.6pt);
\draw[color=black] (3.1,4.2) node {$2i-1$};
\draw [fill=black] (3.,3.) circle (1.6pt);
\draw[color=black] (3,2.8) node {$2i$};
\draw [fill=black] (4.,4.) circle (1.6pt);
\draw[color=black] (4.1,4.2) node {$2i+1$};
\draw [fill=black] (4.,3.) circle (1.6pt);
\draw[color=black] (4,2.8) node {$2i+2$};
\draw [fill=black] (5.,4.) circle (1.6pt);
\draw [fill=black] (5.,3.) circle (1.6pt);

\end{scriptsize}

\draw (4,2.3) node {(A)};
\end{tikzpicture}
     \qquad \qquad  
         \begin{tikzpicture}[line cap=round,line join=round,>=triangle 45,x=1.0cm,y=1.0cm,scale = 1]
\draw [fill=black] (.3,3.5) circle (.6pt);
\draw [fill=black] (.5,3.5) circle (.6pt);
\draw [fill=black] (.7,3.5) circle (.6pt);
\draw [line width=.8pt] (1.,3.5)-- (2.,3.5);

\draw [line width=.8pt] (2.,3.5)-- (3.,4.);
\draw [line width=.8pt] (2.,3.5)-- (3.,3.);
\draw [line width=.8pt] (3.,4.)-- (4.,4.);
\draw [line width=.8pt] (3.,3.)-- (4.,3.);
\draw [line width=.8pt] (4.,4.)-- (4.,3.);
\draw [line width=.8pt] (4.,4.)-- (5.,4.);
\draw [line width=.8pt] (4.,3.)-- (5.,3.);
\draw [line width=.8pt] (5.,4.)-- (5.,3.);
\draw [fill=black] (5.3,4.) circle (.6pt);
\draw [fill=black] (5.5,4.) circle (.6pt);
\draw [fill=black] (5.7,4.) circle (.6pt);
\draw [fill=black] (5.3,3.) circle (.6pt);
\draw [fill=black] (5.5,3.) circle (.6pt);
\draw [fill=black] (5.7,3.) circle (.6pt);
\begin{small}
\draw [fill=black] (1.,3.5) circle (1.6pt);
\draw[color=black] (.9,3.8) node {$V_0$};
\draw [fill=black] (2.,3.5) circle (1.6pt);
\draw[color=black] (1.9,3.8) node {$V_1$};
\end{small}
\begin{scriptsize}
\draw [fill=black] (3.,4.) circle (1.6pt);
\draw[color=black] (3.1,4.2) node {$2i-1$};
\draw [fill=black] (3.,3.) circle (1.6pt);
\draw[color=black] (3,2.8) node {$2i$};
\draw [fill=black] (4.,4.) circle (1.6pt);
\draw[color=black] (4.1,4.2) node {$2i+1$};
\draw [fill=black] (4.,3.) circle (1.6pt);
\draw[color=black] (4,2.8) node {$2i+2$};
\draw [fill=black] (5.,4.) circle (1.6pt);
\draw [fill=black] (5.,3.) circle (1.6pt);


\end{scriptsize}

\draw (3.5,2.3) node {(B)};
\end{tikzpicture}

     \begin{tikzpicture}[line cap=round,line join=round,>=triangle 45,x=1.0cm,y=1.0cm,scale = 1]
\draw [fill=black] (.3,3.5) circle (.6pt);
\draw [fill=black] (.5,3.5) circle (.6pt);
\draw [fill=black] (.7,3.5) circle (.6pt);
\draw [line width=.8pt] (1.,3.5)-- (2.,3.5);

\draw [line width=.8pt] (2.,3.5)-- (3.,4.);
\draw [line width=.8pt] (2.,3.5)-- (3.,3.);
\draw [line width=.8pt] (3.,4.)-- (3.,3.);
\draw [line width=.8pt] (3.,4.)-- (4.,4.);
\draw [line width=.8pt] (3.,3.)-- (4.,3.);
\draw [line width=.8pt] (4.,4.)-- (4.,3.);
\draw [line width=.8pt] (4.,4.)-- (5.,4.);
\draw [line width=.8pt] (4.,3.)-- (5.,3.);
\draw [line width=.8pt] (5.,4.)-- (5.,3.);
\draw [fill=black] (5.3,4.) circle (.6pt);
\draw [fill=black] (5.5,4.) circle (.6pt);
\draw [fill=black] (5.7,4.) circle (.6pt);
\draw [fill=black] (5.3,3.) circle (.6pt);
\draw [fill=black] (5.5,3.) circle (.6pt);
\draw [fill=black] (5.7,3.) circle (.6pt);
\begin{small}
\draw [fill=black] (1.,3.5) circle (1.6pt);
\draw[color=black] (.9,3.8) node {$V_0$};
\draw [fill=black] (2.,3.5) circle (1.6pt);
\draw[color=black] (1.9,3.8) node {$V_1$};
\end{small}
\begin{scriptsize}
\draw [fill=black] (3.,4.) circle (1.6pt);
\draw[color=black] (3.1,4.2) node {$2i+1$};
\draw [fill=black] (3.,3.) circle (1.6pt);
\draw[color=black] (3,2.8) node {$2i+2$};
\draw [fill=black] (4.,4.) circle (1.6pt);
\draw[color=black] (4.1,4.2) node {$2i+3$};
\draw [fill=black] (4.,3.) circle (1.6pt);
\draw[color=black] (4,2.8) node {$2i+4$};
\draw [fill=black] (5.,4.) circle (1.6pt);
\draw [fill=black] (5.,3.) circle (1.6pt);

\draw[color=blue] (2.45,3.1) node {$b_1$};
\draw[color=blue] (2.45,3.9) node {$a_1$};
\draw[color=blue] (1.5,3.35) node {$t_1$};

\end{scriptsize}

\draw (3.5,2.3) node {(C)};
\end{tikzpicture}
\end{center}
\caption{  Panel (A) shows the graph prior to a $\Delta$--Y transformation, Panel (B) shows the graph after the $\Delta$--Y transformation, but prior to the series transformations, and Panel (C) shows the graph after the series transformation.  We note that the graph in the last panel is in the correct configuration to repeat Algorithm \ref{alg:ladder}.  The values of the tails ($t_i$) and edges ($a_i$ and $b_i$) are given for the first 5 transformations in Table~\ref{tab:laddergraph}.}\label{fig:ladderalg}
\end{figure}
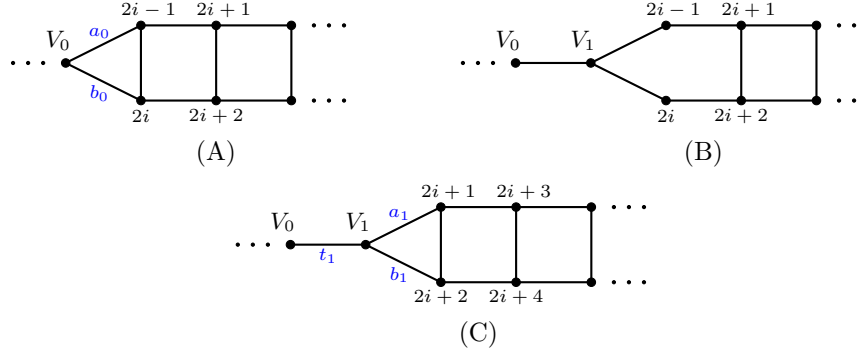

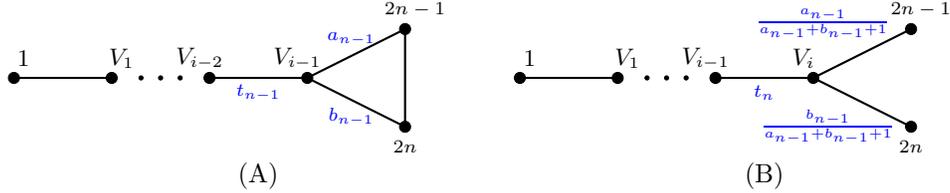
\begin{figure}[ht!]
\begin{center}
    \begin{tikzpicture}[line cap=round,line join=round,>=triangle 45,x=1.0cm,y=1.0cm,scale = 1.3]
\draw [fill=black] (.3,3.5) circle (.6pt);
\draw [fill=black] (.5,3.5) circle (.6pt);
\draw [fill=black] (.7,3.5) circle (.6pt);
\draw [line width=.8pt] (0.,3.5)-- (-1.,3.5);
\draw [line width=.8pt] (1.,3.5)-- (2.,3.5);
\draw [line width=.8pt] (2.,3.5)-- (3.,4.);
\draw [line width=.8pt] (2.,3.5)-- (3.,3.);
\draw [line width=.8pt] (3.,4.)-- (3.,3.);

\begin{small}
\draw [fill=black] (-1.,3.5) circle (1.6pt);
\draw[color=black] (-.9,3.7) node {$1$};
\draw [fill=black] (0.,3.5) circle (1.6pt);
\draw[color=black] (.1,3.7) node {$V_1$};
\draw [fill=black] (1.,3.5) circle (1.6pt);
\draw[color=black] (.9,3.7) node {$V_{i-2}$};
\draw [fill=black] (2.,3.5) circle (1.6pt);
\draw[color=black] (1.9,3.7) node {$V_{i-1}$};

\end{small}
\begin{scriptsize}
\draw[color=blue] (2.45,3.1) node {$b_{n-1}$};
\draw[color=blue] (2.45,3.9) node {$a_{n-1}$};
\draw[color=blue] (1.5,3.35) node {$t_{n-1}$};

\draw [fill=black] (3.,4.) circle (1.6pt);
\draw[color=black] (3.1,4.2) node {$2n-1$};
\draw [fill=black] (3.,3.) circle (1.6pt);
\draw[color=black] (3,2.8) node {$2n$};

\end{scriptsize}

\draw (1.5,2.5) node {(A)};
\end{tikzpicture}
    \qquad  
     \begin{tikzpicture}[line cap=round,line join=round,>=triangle 45,x=1.0cm,y=1.0cm,scale = 1.3]
\draw [fill=black] (.3,3.5) circle (.6pt);
\draw [fill=black] (.5,3.5) circle (.6pt);
\draw [fill=black] (.7,3.5) circle (.6pt);
\draw [line width=.8pt] (0.,3.5)-- (-1.,3.5);
\draw [line width=.8pt] (1.,3.5)-- (2.,3.5);

\draw [line width=.8pt] (2.,3.5)-- (3.,4.);
\draw [line width=.8pt] (2.,3.5)-- (3.,3.);
\begin{small}
\draw [fill=black] (-1.,3.5) circle (1.6pt);
\draw[color=black] (-.9,3.7) node {$1$};
\draw [fill=black] (0.,3.5) circle (1.6pt);
\draw[color=black] (.1,3.7) node {$V_1$};
\draw [fill=black] (1.,3.5) circle (1.6pt);
\draw[color=black] (.9,3.7) node {$V_{i-1}$};
\draw [fill=black] (2.,3.5) circle (1.6pt);
\draw[color=black] (1.9,3.7) node {$V_{i}$};
\end{small}
\begin{scriptsize}
\draw [fill=black] (3.,4.) circle (1.6pt);
\draw[color=black] (3.1,4.2) node {$2n-1$};
\draw [fill=black] (3.,3.) circle (1.6pt);
\draw[color=black] (3,2.8) node {$2n$};
\draw[color=blue] (2.15,3) node {$\frac{b_{n-1}}{a_{n-1}+b_{n-1}+1}$};
\draw[color=blue] (2.1,4.05) node {$\frac{a_{n-1}}{a_{n-1}+b_{n-1}+1}$};
\draw[color=blue] (1.5,3.35) node {$t_{n}$};
\end{scriptsize}

\draw (1.5,2.5) node {(B)};
\end{tikzpicture}
\end{center}
\caption{The final step of the series/$\Delta$--Y transformations in the ladder graph.}\label{fig:ladderalgend}
\end{figure}

 Network transformation algorithms, such as the the one above, have been used  to determine  resistance distances for other structured families of graphs, and likely can be used on many more. We will define here one such family, and present some related open questions in Section \ref{sec:openstuff}. 
\begin{definition}\label{def:kpath}
A K-tree is defined inductively as follows
\begin{enumerate}
    \item The complete graph on $K+1$ vertices is a $K$-tree.
    \item If $G$ is a K-tree, the graph obtained by inserting a vertex adjacent to a clique of $K$ vertices of $G$ is a $K$-tree.
\end{enumerate}
\end{definition}
For a 2-tree, an alternative and more compact definition is: $G$ is a 2-tree on $n$ vertices if $G$ is chordal, has $2n-3$ edges, and $K_4$ is not a subgraph of $G$.
\begin{definition} A linear $K$-tree (or $K$-path) is a $K$-tree in which exactly two vertices have degree $K$. A straight linear $K$-tree is a linear $K$-tree whose adjacency matrix $(a_{ij})$ is banded with bandwidth equal to $k$, i.e., $a_{ij} = 0 $ if $|i-j|>k$.  
\end{definition}
\begin{figure}[ht!]
\begin{center}

\begin{tikzpicture}[line cap=round,line join=round,>=triangle 45,x=1.0cm,y=1.0cm,scale = 1]
\draw [line width=.8pt] (-3.,0.)-- (-2.,0.);
\draw [line width=.8pt] (-2.,0.)-- (-1.,0.);
\draw [line width=.8pt,dotted] (-1.,0.)-- (0.,0.);
\draw [line width=.8pt] (0.,0.)-- (1.,0.);
\draw [line width=.8pt] (1.,0.)-- (2.,0.);
\draw [line width=.8pt] (2.,0.)-- (1.5,0.866025403784435);
\draw [line width=.8pt] (1.5,0.866025403784435)-- (1.,0.);
\draw [line width=.8pt] (1.,0.)-- (0.5,0.8660254037844366);
\draw [line width=.8pt] (0.5,0.8660254037844366)-- (0.,0.);
\draw [line width=.8pt] (-0.5,0.8660254037844378)-- (-1.,0.);
\draw [line width=.8pt] (-1.,0.)-- (-1.5,0.8660254037844385);
\draw [line width=.8pt] (-1.5,0.8660254037844385)-- (-2.,0.);
\draw [line width=.8pt] (-2.,0.)-- (-2.5,0.8660254037844388);
\draw [line width=.8pt] (-2.5,0.8660254037844388)-- (-3.,0.);
\draw [line width=.8pt] (-2.5,0.8660254037844388)-- (-1.5,0.8660254037844385);
\draw [line width=.8pt] (-1.5,0.8660254037844385)-- (-0.5,0.8660254037844378);
\draw [line width=.8pt,dotted] (-0.5,0.8660254037844378)-- (0.5,0.8660254037844366);
\draw [line width=0.8pt] (0.5,0.8660254037844366)-- (1.5,0.866025403784435);
\begin{scriptsize}
\draw [fill=black] (-3.,0.) circle (2.pt);
\draw[color=black] (-3.02279181666165,-0.22431183338253265) node {$1$};
\draw [fill=black] (-2.,0.) circle (2.pt);
\draw[color=black] (-2.0001954862580344,-0.22395896857501957) node {$3$};
\draw [fill=black] (-2.5,0.8660254037844388) circle (2.pt);
\draw[color=black] (-2.5018465162673555,1.100526386601008717) node {$2$};
\draw [fill=black] (-1.5,0.8660254037844385) circle (2.pt);
\draw[color=black] (-1.5081915914412003,1.100526386601008717) node {$4$};
\draw [fill=black] (-1.,0.) circle (2.pt);
\draw[color=black] (-1.0065405614318794,-0.22290037415248035) node {$5$};
\draw [fill=black] (-0.5,0.8660254037844378) circle (2.pt);
\draw[color=black] (-0.4952423962300715,1.100526386601008717) node {$6$};
\draw [fill=black] (0.,0.) circle (2.pt);
\draw[color=black] (-0.03217990699069834,-0.22431183338253265) node {$n-4$};
\draw [fill=black] (0.5,0.8660254037844366) circle (2.pt);
\draw[color=black] (0.4887653934035965,1.103344389715898) node {$n-3$};
\draw [fill=black] (1.,0.) circle (2.pt);
\draw[color=black] (0.9904164234129174,-0.22431183338253265) node {$n-2$};
\draw [fill=black] (1.5,0.866025403784435) circle (2.pt);
\draw[color=black] (1.5692445349621338,1.1042991524908385) node {$n-1$};
\draw [fill=black] (2.,0.) circle (2.pt);
\draw[color=black] (1.993718483431559,-0.22431183338253265) node {$n$};
\end{scriptsize}
\end{tikzpicture}
%
%
\end{center}
\caption{$S_n$, the straight linear 2-tree on $n$ vertices}
\label{fig:2tree}
\end{figure}
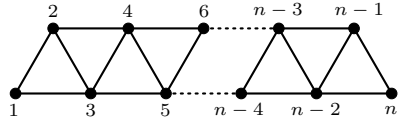


 In \cite{bef}, Barrett and the authors of this paper used network transformations to determine the resistance distance in a straight linear 2-tree with $n$ vertices:

\begin{theorem}~\cite[Th. 20]{bef}\label{thm:sl2t}
Let $S_n$ be the straight linear 2-tree on $n$ vertices labeled as in Figure~\ref{fig:2tree}. Then for any two vertices $u$ and $v$ of $S_n$ with $u < v$, 
\begin{equation}
r_{S_n}(u,v)=\frac{\sum_{i=1}^{v-u} (F_i F_{i+2u-2}-F_{i-1} F_{i+2u-3})F_{2n-2i-2u+1}}{F_{2n-2}}. \label{eq:resdiststraightsum}
\end{equation}\noindent where $F_p$ is the $p$th Fibonacci number.
\end{theorem}

\section{Mathematical techniques for determining the resistance distance}\label{sec:tools}
In this section we detail several different mathematical techniques for determining the resistance distance between two vertices in a graph.  These range from combinatorial techniques to purely computational techniques.  We begin with techniques that are combinatorial in nature and those that rely on special structures or symmetries existing in the graph.  We then cover more numerical techniques concluding with purely programmatic techniques that give an approximate value for resistance distances in a short amount of time and are most appropriate to use for very large, mostly unstructured graphs.  For each of these techniques (except the approximate techniques) we provide an illustrative example. 

\subsection{Spanning Two Forests}\label{sec:span2for}
Our first mathematical technique for determining resistance distance between nodes in a graph relies on the relationship between the number of spanning 2-forests separating two vertices $u$ and $v$ and the resistance distance.  Before stating this relationship we formally define a spanning 2-forest separating vertices $u$ and $v$ in a graph $G$. 
\begin{definition}
Recall that a spanning tree $T$ of a graph $G$ is a subgraph of $G$ that is a tree,  such that $V(T)=V(G)$. A spanning 2-forest is a pair of disjoint trees (called components) which are subgraphs of $G$, such that  $V(T_1) \cup V(T_2) = V(G)$. Given any two vertices $u$ and $v$ of $G$ a spanning 2-forest separating $u$ and $v$ is a spanning 2-forest $F$ where $u$ and $v$ are in distinct tree components.  The number of such forests is denoted by $\F_G(u,v)$. 
\end{definition}

The following theorem~\cite[Th. 4 and (5)]{bapatdvi} gives the relationship between the resistance distance between $u$ and $v$ and the number of spanning 2-forests separating $u$ and $v$.

\begin{theorem}\label{thm:2forests/trees} Given a graph $G$, the resistance distance between vertices $u$ and $v$ is given by \[r_G(u,v)=\dfrac{\F_G(u,v)}{T(G)}=\dfrac{\det L_G(u,v)}{\det L_G(w)},\]
where $w$ is any vertex of $G$.
\end{theorem}
 
\begin{examplei}\label{ex:ladderspanningtrees}
We now return to the example of the ladder graph $L_n$ on $2n$ vertices, with vertices labeled as in Figure \ref{fig:laddergraph}(A). We note that any spanning tree on $L_{n-1}$ can be expanded to a spanning tree of $L_n$ by any of the following 2-edge sets
\[
\{(2n-3,2n-1),(2n-2,2n)\}, \{(2n-3,2n-1),(2n-1,2n)\}, 
\{(2n-2,2n),(2n-1,2n)\}.
\]
The only additional spanning trees of $L_n$ are those which contain the edges $(2n-3,2n-1)$, $(2n-1,2n)$ and $(2n-2,2n)$. These trees, when restricted to the vertices of $L_{n-1}$ are exactly the spanning 2-forests that separate vertices $2n-3$ and $2n-2$.  So, altogether there are $3T_{L_{n-1}} + \F_{L_{n-1}}(2n-3,2n-2)$ spanning trees of $L_n$.  A similar argument demonstrates that the number of spanning 2-forests of $L_n$ that separate vertices $2n-1$ and $2n$ is $2T(L_{n-1})+\F_{L_{n-1}}(2n-3,2n-2)$.  Thus, 
\[
\begin{bmatrix}T_n \\ V_n\end{bmatrix}
=\begin{bmatrix} 3T_{n-1} + V_{n-1}\\
 2T_{n-1} + V_{n-1}\end{bmatrix} = 
\begin{bmatrix} 3 & 1 \\ 2 & 1 \end{bmatrix}
\begin{bmatrix} T_{n-1} \\ V_{n-1} \end{bmatrix}
\]
where  $T_n = T(L_n)$, $V_n = \F_{L_n}(2n-1,2n)$, and $T_1 = V_1 =  1$. 
We note that this system has eigenvalues $2\pm\sqrt{3}$  with corresponding eigenvectors $\begin{bmatrix} \pm 1 & \sqrt{3}\mp 1\end{bmatrix}^T$. Since 
\[
\begin{bmatrix} 1 \\ 1\end{bmatrix} = 
\frac{2+\sqrt{3}}{2\sqrt{3}} \begin{bmatrix} 1 \\ \sqrt{3}-1\end{bmatrix} + 
\frac{2-\sqrt{3}}{2\sqrt{3}}\begin{bmatrix} -1 \\ \sqrt{3}+1\end{bmatrix}
\]
we have
\[
\begin{bmatrix}T_n \\ V_n\end{bmatrix}
=\frac{2+\sqrt{3}}{2\sqrt{3}}\cdot (2+\sqrt{3})^{n-1} \begin{bmatrix} 1 \\ \sqrt{3}-1\end{bmatrix} + 
\frac{2-\sqrt{3}}{2\sqrt{3}}\cdot (2-\sqrt{3})^{n-1}\begin{bmatrix} -1 \\ \sqrt{3}+1\end{bmatrix}
\]
Thus 
\[
T(L_n) = \frac{(2+\sqrt{3})^{n}}{2\sqrt{3}} - 
\frac{(2-\sqrt{3})^{n}}{2\sqrt{3}}
\]
and 
\[
\F_{L_n}(2n-1,2n) = \F_{L_n}(1,2) = 
\frac{(2+\sqrt{3})^{n}}{2\sqrt{3}} (\sqrt{3}-1) + 
\frac{(2-\sqrt{3})^{n}}{2\sqrt{3}}( \sqrt{3}+1)
\]
Theorem \ref{thm:2forests/trees} then yields
\[
r_{L_n}(1,2) = r_{L_n}(2n-1,2n) = 
\frac{(2+\sqrt{3})^{n} (\sqrt{3}-1) + 
(2-\sqrt{3})^{n}( \sqrt{3}+1)}
{(2+\sqrt{3})^{n} - 
(2-\sqrt{3})^{n}}
\]
\end{examplei}
\vspace*{2mm}

In~\cite{BapatWheels} Bapat used Theorem \ref{thm:2forests/trees} to find formulae for resistances between any pair of nodes in a fan graph $F_n$ on $n+1$ vertices (the graph formed when each vertex in a path on $n$ vertices is connected to a ``hub'' vertex $n+1$). In fact, $r_{F_n}(i,n+1)$ was calculated by evaluating determinants of submatrices of $L(G)$, and the combinatorial method of counting separating 2-forests and spanning trees was used to find the resistances between any ``non-central'' nodes: 

\begin{proposition}\cite{BapatWheels}
Let $n \geq 1$ be a positive integer. Then for $i,j = 1, \ldots n$,  
\begin{gather*}
r_{F_n}(i,n+1) = \frac{F_{2(n-i)+1}+F_{2i-1}}{F_{2n}}
\quad \text{and}\\
r_{F_n}(i,j) = 
\frac{F_{2(n-j)+1}(F_{2j-1}-F_{2i-1}) + F_{2i-1}(F_{2(n-i)+1}-F_{2(n-j)+1})}
{F_{2n}},
\end{gather*}
where $F_i$ is the $i$th Fibonacci number.
\end{proposition}

 From Theorem \ref{thm:2forests/trees} one can also determine resistance distance using the idea of 2-separations. Here we review this notion, as explained in ~\cite{bent2tree}.
\begin{definition} A \emph{2-separation} of a graph $G$ is a pair of subgraphs $G_1,G_2$ such that 
\begin{itemize}
    \item $V(G) = V(G_1)\cup V(G_2)$,
    \item $|V(G_1)\cap V(G_2)|=2$, 
    \item $E(G)=E(G_1)\cup E(G_2)$, and 
    \item $E(G_1)\cap E(G_2)=\emptyset$.  
\end{itemize}
The pair of vertices, $V(G_1)\cap V(G_2)$, is called a \emph{2-separator} of $G$.

\end{definition}\begin{definition}
If $G$ is a graph with a 2-separator $\{i,j\}$, and a corresponding 2-separation $G_1,G_2$, then the \emph{2-switch} on $G$ determined by  $\{i,j\}$ is the operation that identifies the copy of $i$ in $G_1$ with the copy of $j$ in $G_2$ and the copy of $j$ in $G_1$ with the copy of $i$ in $G_2$ (see Figure \ref{fig:2_switch}).
\end{definition}

\begin{figure}[ht]
    \centering
    \begin{tikzpicture}
    \begin{scope}[scale = .64]

\draw [line width=.8pt] (1.,2.)-- (1.,1.);
\draw [line width=.8pt] (1.,2.)-- (2.,2.);
\draw [line width=.8pt] (2.,2.)-- (2.,1.);
\draw [line width=.8pt] (2.,1.)-- (1.,1.);
\draw [line width=.8pt] (2.,2.)-- (3.,2.);
\draw [line width=.8pt] (3.,2.)-- (3.,1.);
\draw [line width=.8pt] (3.,1.)-- (2.,1.);
\draw [line width=.8pt] (3.,2.)-- (4.,2.);
\draw [line width=.8pt] (4.,2.)-- (4.,1.);
\draw [line width=.8pt] (4.,1.)-- (3.,1.);
\draw [line width=.8pt] (4.,2.)-- (5.,2.);
\draw [line width=.8pt] (5.,2.)-- (5.,1.);
\draw [line width=.8pt] (5.,1.)-- (4.,1.);
\begin{scriptsize}
\draw [fill=black] (1.,2.) circle (1.6pt);
\draw [fill=black] (1.,1.) circle (1.6pt);
\draw [fill=black] (2.,2.) circle (1.6pt);
\draw [fill=black] (2.,1.) circle (1.6pt);
\draw [fill=black] (3.,2.) circle (1.6pt);
\draw [fill=black] (3.,1.) circle (1.6pt);
\draw[color=black] (3,0.7) node {$i$};
\draw [fill=black] (4.,2.) circle (1.6pt);
\draw[color=black] (4.,2.3) node {$j$};
\draw [fill=black] (4.,1.) circle (1.6pt);
\draw [fill=black] (5.,2.) circle (1.6pt);
\draw [fill=black] (5.,1.) circle (1.6pt);
\end{scriptsize}

\node at (5.5,1.5) {$\to$};

    \end{scope}
    \begin{scope}[xshift = .26 \textwidth, scale = .64]
\draw [line width=.8pt] (1.,2.)-- (1.,1.);
\draw [line width=.8pt] (1.,2.)-- (2.,2.);
\draw [line width=.8pt] (2.,2.)-- (2.,1.);
\draw [line width=.8pt] (2.,1.)-- (1.,1.);
\draw [line width=.8pt] (2.,2.)-- (3.,2.);
\draw [line width=.8pt] (3.,2.)-- (3.,1.);
\draw [line width=.8pt] (3.,1.)-- (2.,1.);
\draw [line width=.8pt] (3.,2.)-- (4.,2.);
\draw [line width=.8pt] (4.8,2.)-- (4.8,1.);
\draw [line width=.8pt] (4.8,1.)-- (3.8,1.);
\draw [line width=.8pt] (4.8,2.)-- (5.8,2.);
\draw [line width=.8pt] (5.8,2.)-- (5.8,1.);
\draw [line width=.8pt] (5.8,1.)-- (4.8,1.);
\begin{scriptsize}
\draw [fill=black] (1.,2.) circle (1.6pt);
\draw [fill=black] (1.,1.) circle (1.6pt);
\draw [fill=black] (2.,2.) circle (1.6pt);
\draw [fill=black] (2.,1.) circle (1.6pt);
\draw [fill=black] (3.,2.) circle (1.6pt);
\draw[color=black] (2.9,0.75) node {$i$};
\draw [fill=black] (4.,2.) circle (1.6pt);
\draw[color=black] (4.2,2.25) node {$j$};

\draw [fill=black] (3.8,1.) circle (1.6pt);
\draw[color=black] (3.55,0.75) node {$i$};
\draw [fill=black] (4.8,2.) circle (1.6pt);
\draw[color=black] (5.05,2.25) node {$j$};
\draw [fill=black] (4.8,1.) circle (1.6pt);
\draw [fill=black] (5.8,2.) circle (1.6pt);
\draw [fill=black] (5.8,1.) circle (1.6pt);
\end{scriptsize}

\node at (6.5,1.5) {$\to$};

    \end{scope}
    %
    \begin{scope}[xshift = .575\textwidth, scale = .64]
\draw [line width=.8pt] (1.,2.)-- (1.,1.);
\draw [line width=.8pt] (1.,2.)-- (2.,2.);
\draw [line width=.8pt] (2.,2.)-- (2.,1.);
\draw [line width=.8pt] (2.,1.)-- (1.,1.);
\draw [line width=.8pt] (2.,2.)-- (3.,2.);
\draw [line width=.8pt] (3.,2.)-- (3.,1.);
\draw [line width=.8pt] (3.,1.)-- (2.,1.);
\draw [line width=.8pt] (3.,2.)-- (4.,2.);
\draw [line width=.8pt] (4.8,2.)-- (4.8,1.);
\draw [line width=.8pt] (4.8,1.)-- (3.8,1.);
\draw [line width=.8pt] (4.8,1.)-- (4.8,0.);
\draw [line width=.8pt] (3.8,1.)-- (3.8,0.);
\draw [line width=.8pt] (3.8,0.)-- (4.8,0.);
\begin{scriptsize}
\draw [fill=black] (1.,2.) circle (1.6pt);
\draw [fill=black] (1.,1.) circle (1.6pt);
\draw [fill=black] (2.,2.) circle (1.6pt);
\draw [fill=black] (2.,1.) circle (1.6pt);
\draw [fill=black] (3.,2.) circle (1.6pt);
\draw [fill=black] (3.,1.) circle (1.6pt);
\draw[color=black] (2.9,0.75) node {$i$};
\draw [fill=black] (4.,2.) circle (1.6pt);
\draw[color=black] (4.2,2.25) node {$j$};

\draw [fill=black] (3.8,1.) circle (1.6pt);
\draw[color=black] (3.55,0.75) node {$j$};
\draw [fill=black] (4.8,2.) circle (1.6pt);
\draw[color=black] (5.05,2.25) node {$i$};
\draw [fill=black] (4.8,1.) circle (1.6pt);
\draw [fill=black] (4.8,0.) circle (1.6pt);
\draw [fill=black] (3.8,0.) circle (1.6pt);
\end{scriptsize}

\node at (5.5,1.5) {$\to$};

    \end{scope}
    \begin{scope}[xshift = .84\textwidth, scale = .63]
\draw [line width=.8pt] (1.,2.)-- (1.,1.);
\draw [line width=.8pt] (1.,2.)-- (2.,2.);
\draw [line width=.8pt] (2.,2.)-- (2.,1.);
\draw [line width=.8pt] (2.,1.)-- (1.,1.);
\draw [line width=.8pt] (2.,2.)-- (3.,2.);
\draw [line width=.8pt] (3.,2.)-- (3.,1.);
\draw [line width=.8pt] (3.,1.)-- (2.,1.);
\draw [line width=.8pt] (3.,2.)-- (4.,2.);
\draw [line width=.8pt] (4.,2.)-- (4.,1.);
\draw [line width=.8pt] (4.,1.)-- (3.,1.);
\draw [line width=.8pt] (4.,1.)-- (4.,0.);
\draw [line width=.8pt] (3.,1.)-- (3.,0.);
\draw [line width=.8pt] (3.,0.)-- (4.,0.);
\begin{scriptsize}
\draw [fill=black] (1.,2.) circle (1.6pt);
\draw [fill=black] (1.,1.) circle (1.6pt);
\draw [fill=black] (2.,2.) circle (1.6pt);
\draw [fill=black] (2.,1.) circle (1.6pt);
\draw [fill=black] (3.,2.) circle (1.6pt);
\draw [fill=black] (3.,1.) circle (1.6pt);
\draw[color=black] (2.75,0.75) node {$i$};
\draw [fill=black] (4.,2.) circle (1.6pt);
\draw[color=black] (4.25,2.25) node {$j$};

\draw [fill=black] (4.,1.) circle (1.6pt);
\draw [fill=black] (4.,0.) circle (1.6pt);
\draw [fill=black] (3.,0.) circle (1.6pt);
\end{scriptsize}

\end{scope}
    \end{tikzpicture}
    \caption{2-switch }
    \label{fig:2_switch}
\end{figure}
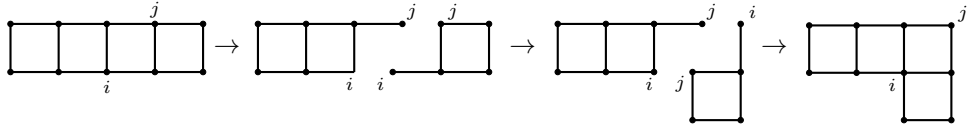

\begin{definition}\label{def:identify}
If $i,j$ are vertices of $G$, then $G/ij$ is the graph obtained from $G$ by identifying vertices $i$ and $j$ into a vertex we denote by $ij$.  In the identification, any edge $\{u,i\}$ or $\{u,j\}$ with $u \ne i,j$ is replaced by an edge $\{u,ij\}$. (So if $i$ and $j$ have a common neighbor $v$, then there is a double edge from $v$ to $ij$ in the new graph.)  Any edge $\{u,v\}$ with neither $u$ nor $v$ equal to $i$ or $j$ remains.  If $\{i,j\}$ is an edge it disappears. 
\end{definition}

The following two theorems from  from~\cite{bent2tree} are especially useful and were instrumental in showing the results in~\cite{swim2019}.
\begin{theorem}\label{thm:2sep2switch}
Let $G$ be a graph with a 2-separator $\{i,j\}$ and $G'$ the graph obtained by performing the 2-switch determined by $i$ and $j$.  Then 
\[T(G)= T(G').\]
Moreover, if $u$ and $v$ are both in $G_1$,  or if $u$ and $v$ are both in $G_2$ and neither is equal to $i$ or $j$, then 
\begin{align*}
\F_G(u,v) &= \F_{G'}(u,v) \quad \text{and}\\
r_G(u,v) &= r_{G'}(u,v).
\end{align*}
If  $u$ and $v$ are both in $G_2$ and exactly one of them (say, $u$) is equal to  $i$ or $j$, we have
\begin{equation}\label{eq:thm10_labels}
\F_G(i,v) = \F_{G'}(j,v), \quad   \F_G(j,v) = \F_{G'}(i,v), 
\end{equation}
\begin{equation*}
    r_G(i,v) = r_{G'}(j,v),\quad  \text{and} \quad r_G(j,v) = r_{G'}(i,v).
\end{equation*}
\end{theorem}
\begin{theorem}\label{thm:formulation1}
Let $G$ be a graph with a 2-separation, with $i,j$ the two vertices separating the graph, and $G_1, G_2$ the two graphs of the separation.  Let $u\in V(G_1)$ and $v\in V(G_2)$.  Then
\begin{align*}
\F_G(u,v) =& \F_{G_1/ij}(u,ij)T(G_2) +\F_{G_2/ij}(v,ij)T(G_1)\\& + \F_{G_1}(u,i)\F_{G_2}(v,j)+\F_{G_1}(u,j)\F_{G_2}(v,i) \\& - 2\F_{G_1}(u,\{i,j\})\F_{G_2}(v,\{i,j\}),
\end{align*}
where, $\F_H(a,\{b,c\})$ is the number of spanning 2-forests separating the vertex $a$ from the pair of vertices $b$ and $c$ in graph $H$.  
\end{theorem}


\begin{examplei}
Again, we return to the example of the ladder graph $L_n$ on $2n$ vertices. We will compute $\F_{L_n}(1,2n)$ and $\F_{L_n}(1,2n-1)$ using Theorem \ref{thm:formulation1}, using difference equations as in Example \ref{ex:ladderspanningtrees} (above). We set 
\[
\begin{array}{c}
F_n = \F_{L_n}(1,2n), \\
\tilde{F}_n = \F_{L_n}(1,2n-1)\\
A_n = \F_{L_n/2n-1\text{-}2n}(1,2n-1\text{-}2n).
\end{array}
\]
Applying Theorem \ref{thm:formulation1} to find $F_n$, $\tilde{F}_n$, and $A_n$ (with $i = 2n-3$, $j = 2n-2$ in each case) gives
\[
\begin{array}{l}
F_n= A_{n-1}+ 2T_{n-1} + \tilde{F}_{n-1} + 2F_{n-1} \\
\tilde{F}_n= A_{n-1} + 2T_{n-1} + 2\tilde{F}_{n-1} + F_{n-1}\\
A_n=  A_{n-1}+T_{n-1}+F_{n-1}
+\tilde{F}_{n-1} \\

\end{array}
\]

Note that $F_2 = 4$, $\tilde{F}_2 = 3$, $A_2 = 2$,  and $T_2 = 4$. We have the following linear system: 
\[
\begin{bmatrix}
2 & 1 & 1  & 2 & 0 \\
1 & 2 & 1 & 2  & 0 \\
1 &1 & 1 &  1 & 0 \\
0 & 0  & 0 & 3 & 1 \\
0 & 0  & 0 & 2 & 1 
\end{bmatrix}
\begin{bmatrix}
F_{n-1} \\ \tilde{F}_{n-1} \\ A_{n-1} \\ T_{n-1} \\ V_{n-1}
\end{bmatrix}
=\begin{bmatrix}
F_{n} \\ \tilde{F}_{n} \\ A_{n} \\ T_{n} \\ V_{n}
\end{bmatrix}
\]

This matrix has only 3 linearly independent eigenvectors, so we cannot use our previous methods to find general formulae for the sequences in question. We can, however, use matrix multiplication to derive the first few terms, 

\[
\begin{bmatrix}
F_{2} \\ \tilde{F}_{2} \\ A_{2} \\ T_{2} \\ V_{2}
\end{bmatrix} = 
\begin{bmatrix}
4 \\ 3 \\ 2 \\4 \\ 3
\end{bmatrix}, \quad 
\begin{bmatrix}
F_{3} \\ \tilde{F}_{3} \\ A_{3} \\ T_{3} \\ V_{3}
\end{bmatrix}  
=\begin{bmatrix}
21\\ 20\\ 13\\ 15\\ 11
\end{bmatrix}, \quad 
\begin{bmatrix}
F_{4} \\ \tilde{F}_{4} \\ A_{4} \\ T_{4} \\ V_{4}
\end{bmatrix} 
=\begin{bmatrix}
105\\ 104\\ 69\\ 56\\ 41
\end{bmatrix}, \quad
\begin{bmatrix}
F_{5} \\ \tilde{F}_{5} \\ A_{5} \\ T_{5} \\ V_{5}
\end{bmatrix} 
=\begin{bmatrix}
495\\ 494\\ 334\\ 209\\ 153
\end{bmatrix}.
\]
And thus, 
\[
\begin{array}{| c | c | c | c |}
\hline
n & r_{L_n}(1,2) & r_{L_n}(1,2n-1) & r_{L_n}(1,2n) \\\hline
2&3/4 & 3/4 & 1\\
3&11/15 & 20/15& 21/15 \\
4&41/56 & 104/56& 105/56\\
5 & 153/209 & 494/209 & 495/209\\\hline
\end{array}
\]

\end{examplei}
\vspace*{3mm}

In~\cite{swim2019} Barrett and the authors used the method of spanning 2-forests and 2-switches to generalize the formulas for a straight linear 2-tree to linear 2-trees with any number of bends.  The definition is as follows:
\begin{definition}\label{def:lin2treestb}
We define the graph $G_n$ with $V(G_n)= V(S_n)$ and $E(G_n) = (E(S_n) \cup \{k-1,k+2\})\setminus(\{k,k+2\})$ to be the bent linear 2-tree with bend at vertex $k$.
 See Figure~\ref{fig:bent}.
\end{definition}

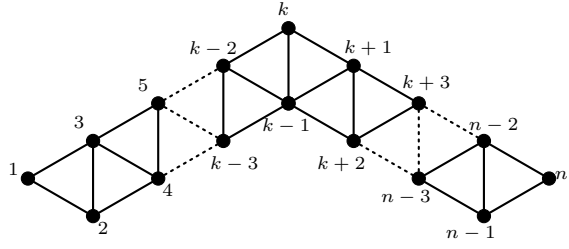
\begin{figure}
\begin{center}
\begin{tikzpicture}[line cap=round,line join=round,>=triangle 45,x=1.0cm,y=1.0cm,scale = 1]
\draw [line width=.8pt] (-5.464101615137757,2.)-- (-4.598076211353318,1.5);
\draw [line width=.8pt] (-4.598076211353318,1.5)-- (-4.598076211353318,2.5);
\draw [line width=.8pt] (-4.598076211353318,2.5)-- (-3.732050807568879,2.);
\draw [line width=.8pt] (-3.732050807568879,2.)-- (-3.7320508075688785,3.);
\draw [line width=.8pt,dotted] (-3.7320508075688785,3.)-- (-2.8660254037844393,2.5);
\draw [line width=.8pt] (-2.8660254037844393,2.5)-- (-2.866025403784439,3.5);
\draw [line width=.8pt] (-2.866025403784439,3.5)-- (-2.,3.);
\draw [line width=.8pt] (-2.,3.)-- (-2.,4.);
\draw [line width=.8pt] (-2.,4.)-- (-1.1339745962155612,3.5);
\draw [line width=.8pt] (-1.1339745962155612,3.5)-- (-2.,3.);
\draw [line width=.8pt] (-2.,3.)-- (-1.1339745962155616,2.5);
\draw [line width=.8pt] (-1.1339745962155616,2.5)-- (-1.1339745962155612,3.5);
\draw [line width=.8pt] (-1.1339745962155612,3.5)-- (-0.2679491924311225,3.);
\draw [line width=.8pt] (-0.2679491924311225,3.)-- (-1.1339745962155616,2.5);
\draw [line width=.8pt,dotted] (-1.1339745962155616,2.5)-- (-0.26794919243112303,2.);
\draw [line width=.8pt,dotted] (-0.26794919243112303,2.)-- (-0.2679491924311225,3.);
\draw [line width=.8pt,dotted] (-0.2679491924311225,3.)-- (0.5980762113533165,2.5);
\draw [line width=.8pt] (0.5980762113533165,2.5)-- (-0.26794919243112303,2.);
\draw [line width=.8pt] (-0.26794919243112303,2.)-- (0.5980762113533152,1.5);
\draw [line width=.8pt] (0.5980762113533152,1.5)-- (0.5980762113533165,2.5);
\draw [line width=.8pt] (0.5980762113533165,2.5)-- (1.464101615137755,2.);
\draw [line width=.8pt] (1.464101615137755,2.)-- (0.5980762113533152,1.5);
\draw [line width=.8pt] (-2.,4.)-- (-2.866025403784439,3.5);
\draw [line width=.8pt,dotted] (-2.866025403784439,3.5)-- (-3.7320508075688785,3.);
\draw [line width=.8pt] (-3.7320508075688785,3.)-- (-4.598076211353318,2.5);
\draw [line width=.8pt] (-4.598076211353318,2.5)-- (-5.464101615137757,2.);
\draw [line width=.8pt] (-4.598076211353318,1.5)-- (-3.732050807568879,2.);
\draw [line width=.8pt,dotted] (-3.732050807568879,2.)-- (-2.8660254037844393,2.5);
\draw [line width=.8pt] (-2.8660254037844393,2.5)-- (-2.,3.);
\begin{scriptsize}
\draw [fill=black] (-2.,4.) circle (2.5pt);
\draw[color=black] (-2.06648828953202,4.259331085745072) node {$k$};
\draw [fill=black] (-1.1339745962155612,3.5) circle (2.5pt);
\draw[color=black] (-0.9407359844212153,3.7428094398707032) node {$k+1$};
\draw [fill=black] (-2.,3.) circle (2.5pt);
\draw[color=black] (-2.0267558552339913,2.6989231016718021) node {$k-1$};
\draw [fill=black] (-2.866025403784439,3.5) circle (2.5pt);
\draw[color=black] (-3.,3.73) node {$k-2$};
\draw [fill=black] (-2.8660254037844393,2.5) circle (2.5pt);
\draw[color=black] (-2.72,2.2) node {$k-3$};
\draw [fill=black] (-1.1339745962155616,2.5) circle (2.5pt);
\draw[color=black] (-1.3,2.2) node {$k+2$};
\draw [fill=black] (-2.,3.) circle (2.5pt);
\draw [fill=black] (-0.2679491924311225,3.) circle (2.5pt);
\draw[color=black] (-0.17,3.27) node {$k+3$};
\draw [fill=black] (-3.7320508075688785,3.) circle (2.5pt);
\draw[color=black] (-3.9339127015393545,3.2395319387623434) node {$5$};
\draw [fill=black] (-3.732050807568879,2.) circle (2.5pt);
\draw[color=black] (-3.6028090823891175,1.8488967383313488) node {$4$};
\draw [fill=black] (-0.26794919243112303,2.) circle (2.5pt);
\draw[color=black] (-0.4374584833128556,1.7488967383313488) node {$n-3$};
\draw [fill=black] (-4.598076211353318,2.5) circle (2.5pt);
\draw[color=black] (-4.78153796656396,2.7494985824199927) node {$3$};
\draw [fill=black] (-4.598076211353318,1.5) circle (2.5pt);
\draw[color=black] (-4.463678492179733,1.345619237222989) node {$2$};
\draw [fill=black] (-5.464101615137757,2.) circle (2.5pt);
\draw[color=black] (-5.655651521120585,2.1270237784175476) node {$1$};
\draw [fill=black] (0.5980762113533165,2.5) circle (2.5pt);
\draw[color=black] (0.7280262560959774,2.709766148121964) node {$n-2$};
\draw [fill=black] (0.5980762113533152,1.5) circle (2.5pt);
\draw[color=black] (0.4234109264777597,1.2721075267550078) node {$n-1$};
\draw [fill=black] (1.464101615137755,2.) circle (2.5pt);
\draw[color=black] (1.628628100184621,2.0475589098214906) node {$n$};
\end{scriptsize}
\end{tikzpicture}
\end{center}
\caption{A linear 2-tree with $n$ vertices and single bend at vertex $k$.}
\label{fig:bent}
\end{figure}

In essence, performing a bend operation at vertex $k$  results in vertex $k-1$ having degree 5, vertex $k$ having degree $3$ and all other vertices having the same degrees as before. 

The following result is the main result from~\cite{swim2019}.

\begin{theorem}~\cite[Th. 3.1]{swim2019}\label{cor:main2}
 Given a bent linear 2-tree with $n$ vertices, and $p = p_1 + p_2 + p_3$ single bends located at nodes $k_1, k_2,  \ldots, k_p$ and $k_1 < k_2 < \cdots < k_{p-1} < k_p$ the number of spanning 2-forests separating nodes $u$ and $v$ where ${k_{p_1}< u \leq k_{p_1+1}}$ and $k_{p_1+p_2}< v \leq k_{p_1+p_2+1}$ is given by
  	 \begin{multline}\label{eq:genericform2}
	 \F_G(u,v)=\F_{S_n}(u,v)\\-\sum_{j=p_1+1}^{p_1+p_2}\Big[F_{k_j-3}F_{k_j}+
	 2\sum_{i=p_1+1}^{j-1}[(-1)^{k_j-k_{i}+j-i}F_{k_{p_1+i}}F_{k_{p_1+i}-3}]+2(-1)^{j+u+k_j}F_{u-1}^2\Big]\cdot\\
	 \Big[F_{n-k_j+2}F_{n-k_j-1}+2(-1)^{v-k_j}F_{n-v}^2\Big],
	 \end{multline}
 and the resistance distance between nodes $u$ and $v$ is given by 
	 \begin{multline}\label{eq:genericformres}
	 r_G(u,v)=r_{S_n}(u,v)\\
	 -\sum_{j=p_1+1}^{p_1+p_2}\Big[F_{k_j-3}F_{k_j}+2\sum_{i=p_1+1}^{j-1}[(-1)^{k_j-k_i+j-i}F_{k_i}F_{k_i-3}]+2(-1)^{j+u+k_j}F_{u-1}^2\Big]\cdot\\
	 \Big[F_{n-k_j+2}F_{n-k_j-1}+2(-1)^{v-k_j}F_{n-v}^2\Big]/F_{2n-2}.
	 \end{multline}
 \end{theorem}

In~\cite{MarkK} the method of spanning 2-forests and 2-switches was used to determine resistance distance in a generalized flower graph.
\begin{definition}
Let $G$ be a graph, $x, y$ be two distinct vertices of $G$, and $n\geq 3$. A generalized flower of $G$, written $F_n(G, x, y)$, is the graph obtained by taking $n$ vertex disjoint copies $G_1, G_2,\ldots G_n$ of the base graph, and associating $x_{i-1}$, the marked vertex $x$ in $G_{i-1}$, with $y_i$ for $1< i < n$ and $x_1$ with $y_n$.  A complete flower graph is a flower graph where $G = K_m$ for some $m \geq 3$, and is denoted by$F_n(K_m)$.  See Figure~\ref{fig:flowerpic}
\end{definition}
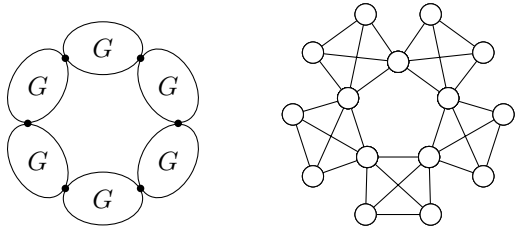
\begin{figure}[ht!]
\begin{center}
\begin{tikzpicture}
        \tikzstyle{every node}=[circle, fill=black, inner sep=1pt]
        \draw{ \foreach \x in {0,60,120,180,240,300} {
        (\x:1)node{}
        }
        };
        \foreach \x in {30,90,150,210,270,330}{
        \draw[rotate around={\x: (\x:1)},](\x:1) ellipse(10 pt and 15 pt);
        }
        \tikzstyle{every node}=[circle, draw=none, fill=white, minimum width = 8 pt, inner sep=1pt]
        \draw \foreach \x in {30,90,150,210,270,330}{
        (\x:1)node[]{$G$}
        };
        \end{tikzpicture}\quad\quad\quad
        \begin{tikzpicture}[scale=.7]
        \tikzstyle{every node}=[circle, draw, fill=white, minimum width = 8 pt, inner sep=1pt]
            \draw \foreach \x in {18, 90, 162, 234, 306} {
                (\x:1)node{}--(\x+72:1)node{}
                (\x:1)node{}--(\x+18:2)node{}
                (\x+72:1)node{}--(\x+54:2)node{}
                (\x+18:2)node{}--(\x+54:2)node{}
                (\x:1)node{}--(\x+54:2)node{}
                (\x+72:1)node{}--(\x+18:2)node{}
            };
        \end{tikzpicture}
        \end{center}
        \caption{Two examples of flower graphs.  On the left is a generic generalized flower graph with six copies of graph $G$.  On the right is a complete flower graph, $F_5(K_4)$ created on 5 copies of $K_4$.}\label{fig:flowerpic}
\end{figure}
In particular, the following nice result for resistance distance in a complete flower graph can be determined using 2-separators.
\begin{theorem}[\cite{MarkK}]
Let $G$ be a complete flower $F_n(K_m)$ and $u$ and $v$ be vertices in $G$. Let $I$ be the set of associated vertices connecting each copy of $K_m$, then
\begin{align*}
r_{F_n(K_m)}(u, v) = \frac{2d(n-d)}{mn}\quad\quad&\quad\text{if both $u,v \in I$}\\
r_{F_(K_m)}(u,v)=\frac{2d}{m}-\frac{(2d-1)^2}{2mn}\quad\quad&\quad\text{if one of $u,v \in I$}\\
r_{F_(K_m)}(u,v)=\frac{2d}{m}-\frac{(2d-1)^2}{mn}\quad\quad&\quad\text{if neither of $u,v \in I$}
\end{align*}
    Where $d$ is the number of flower petals separating $u$ and $v$ including the petals containing $u$, and $v$.

 \end{theorem}
The proof of this theorem, and others in the paper, including a more complicated result for the generalized flower graph, is shown by creating a 2-separation using vertices in $I$ such that $u$ and $v$ are in the same component, which contains $d$ petals.  This component is then further 2-separated using each of the petals, and a version of Theorem~\ref{thm:2sep2switch} is used to obtain the result.  For complete details see~\cite{MarkK}.


\subsection{Recursion Techniques}
In 2013 Yang and Klein~\cite{YangKlein} provided a technique for determining the resistance distance in a new graph from the resistance distance in a known graph via the changing of an edge weight on a single edge (including changing the weight from zero to non-zero and vice versa).  Their main result is as follows:
\begin{theorem}\label{thm:recur}
     Let $\Omega$ and $\Omega'$ be resistance distance functions for edge-weighted connected graphs $G$ and $G'$ which are the same, except for the weights $w$ and $w'$ on an edge $e$ with end vertices $i$ and $j$.  Then for any $p,q \in V(G)=V(G')$,
     \[\Omega'(p,q) = \Omega(p,q) - \dfrac{\delta[\Omega(p,i) + \Omega(q,j)-\Omega(p,j)-\Omega(q,i)]^2}{4[1+\delta\Omega(i,j)]},\]
     where $\delta = w'-w$.
     
\end{theorem}

This technique works especially well in the case that graph $H$ can be built (or subtracted from) graph $G$ in a limited number of edge weight changes.  An example of this is the derivation of the formula for the fan graph $F_n$ (defined in Subsection~\ref{sec:span2for}) from the wheel graph $W_n$.  In this case the change from a fan graph to a wheel graph involves changing the weight of a single edge.  For full details of this example, and several similar examples we refer the reader to~\cite{YangKlein}.  An obvious advantage of this approach is that a universal formula holds for all pairs of vertices in the graph.  

When a large number of edge weights must be changed, or in other words a large number of edges must be added or subtracted, multiple resistance values must be tracked through each recursive step.  We illustrate this need to track a large number of values by returning to the ladder graph in Example \ref{ex:laddertransform}.  \\

\begin{examplei} We begin with a path graph $G$ with the vertices ordered as $1,3,5,\ldots,2n-1, 2n, 2n-2, \ldots, 2$.  The resistance distance between two vertices in this path graph is given by
\[r_G(p,q)=\begin{cases} |p-q|/2&\text{if $p+q$ is even}\\
2n-\lfloor (p+q)/2\rfloor &\text{if $p+q$ is odd}\end{cases}.\]
We first modify $G$ by adding an edge of weight zero between vertices $1$ and $2$ which does not change the resistance distances in the graph (we note that this corresponds to adding an edge with infinite resistance).  The graph cycle graph $G'$ is then created by changing the weight of the edge between vertices $1$ and $2$ to have unit weight.  Hence
\[r_{G'}(p,q) = r_{G}(p,q) -\dfrac{[r_G(p,1)+r_G(q,2)-r_G(p,2)-r_G(q,1)]^2}{4[1+r_G(1,2)]}.\]

We then proceed to build the ladder graph by first adding an edge of weight zero between vertices $k$ and $k+1$ and then using the recursion formula to determine the resistance distance on the graph with the weight of the new edge equal to one.  This technique is illustrated in Figure~\ref{fig:recursive_ladder}.

To illustrate the challenges associated with technique, suppose that we are interested only in the resistance distance between vertex $1$ and vertex $2n$ in the final ladder, and we wish to use this technique to determine this resistance.   When changing the weight of the edge $(1,2)$ from $0$ to $1$ to determine $r_{G'}(1,2n)$ in the new graph we need the values of $r_G(1,2)$, $r_G(1,2n)$, and $r_G(2,2n)$.  Once these have been determined and we create $G''$ by adding an edge $(3,4)$ and changing the edge weight from $0$ to $1$
we need the values $r_{G'}(1,2n)$, $r_{G'}(1,3)$, $r_{G'}(1,4)$, $r_{G'}(3,2n)$, $r_{G'}(4,2n)$, and $r_{G'}(1,2n)$.  These values must be determined from the original graph $G$.  In fact to determine the resistance $r_L(1,2n)$, where $L$ denotes the graph one must calculate $r_G(1,j)$ where $2 \leq j\leq 2n$, $r_G(j,2n)$ where $2\leq j \leq 2n-1$, and $r(j,j+1)$ where $j$ is odd and $3 \leq j \leq 2n-3$, or in short $5n-5$ values.  By a similar argument, in the next graph (i.e., the cycle) one must calculate $r_{G'}(1,j)$ where $3 \leq j\leq 2n$, $r_G(j,2n)$ where $3\leq j \leq 2n-1$, and $r(j,j+1)$ where $j$ is odd and $3 \leq j \leq 2n-3$, or in short $5n-7$ values.  In total $(5n^2+5n)/2$ resistance values must be computed to compute $r_L(1,2n)$.
\end{examplei}

\begin{figure}\begin{center}
\begin{tikzpicture}
        \matrix [matrix of math nodes,left delimiter=(,right delimiter=)] (m)
        {
            0 &r_{12} &r_{13} & r_{14} & r_{15} & r_{16}& \ldots &r_{1n} \\               
            r_{21} &0 &r_{23} & r_{24} & r_{25} & r_{26}& \ldots &r_{2n}\\             
            r_{31} &r_{32} & 0 & r_{34} & r_{35} & r_{36}& \ldots &r_{3n}\\ 
            r_{41} &r_{42} & r_{43} & 0 & r_{45} & r_{46}& \ldots &r_{4n}\\ 
            r_{51} &r_{52} & r_{53} & r_{54} & 0 & r_{56}& \ldots &r_{5n}\\ 
            \vdots & \vdots & \vdots & \vdots &\vdots&\vdots &\ldots &\vdots\\               
            r_{n-1,1} &r_{n-1,2} & r_{n-1,3} & r_{n-1,4} & r_{n-1,5} & r_{n-1,6} & \ldots &r_{n-1,n}\\             
            r_{n1} &r_{n2} & r_{n3} & r_{n4} & r_{n5} & r_{n6} & \ldots &0\\             
        };  
        \draw[color=blue] (m-3-4.north west) -- (m-3-4.north east)-- (m-3-4.south east) --(m-3-4.south west) -- (m-3-4.north west);
        \draw[color=red] (m-1-2.north west) -- (m-1-2.north east) -- (m-1-2.south east) --  (m-1-2.south west) -- (m-1-2.north west);

        \draw[color=blue] (m-1-3.north west) -- (m-1-4.north east) -- (m-1-4.south east) --  (m-1-3.south west) -- (m-1-3.north west);
        
        \draw[color=green] (m-1-5.north west) -- (m-1-6.north east) -- (m-1-6.south east) --  (m-1-5.south west) -- (m-1-5.north west);

        \draw[color=blue] (m-3-8.north west) -- (m-3-8.north east) -- (m-4-8.south east) -- (m-4-8.south west) --(m-3-8.north west);
        \draw[color=green] (m-5-8.north west) -- (m-5-8.north east) -- (m-5-8.south east) -- (m-5-8.south west) --(m-5-8.north west);

        \draw[color=black] (m-1-8.north west) -- (m-1-8.north east) -- (m-1-8.south east) -- (m-1-8.south west) --(m-1-8.north west);

        \draw[color=green] (m-5-6.north west) -- (m-5-6.north east)-- (m-5-6.south east) --(m-5-6.south west) -- (m-5-6.north west);
        
    \end{tikzpicture}
    \end{center}
    \caption{The values of $r_G$ that must be calculated in order to determine $r_L(1,2n)$. The red values are needed in the first step, the blue values the second, and the green values the third step.  }
    \label{fig:redmatrix}
    \end{figure}
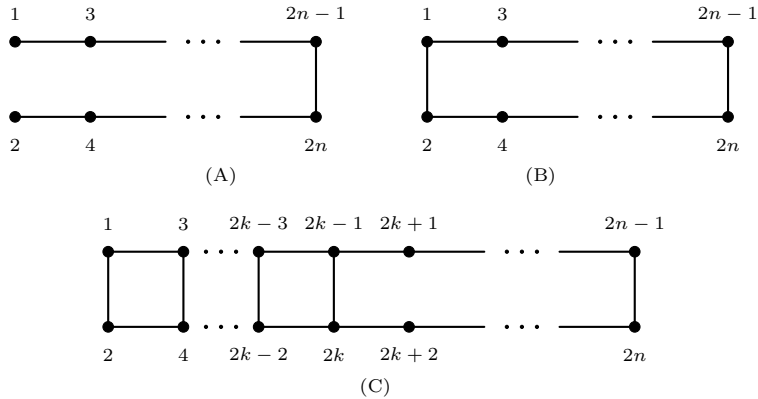
\begin{figure}
    \centering

\begin{tikzpicture}[line cap=round,line join=round,>=triangle 45,x=1.0cm,y=1.0cm]
\draw [line width=.8pt] (-2.,4.)-- (-1.,4.);
\draw [line width=.8pt] (-1.,3.)-- (-2.,3.);
\draw [line width=.8pt] (1.,4.)-- (2.,4.);
\draw [line width=.8pt] (2.,4.)-- (2.,3.);
\draw [line width=.8pt] (2.,3.)-- (1.,3.);
\draw [line width=.8pt] (-1.,4.)-- (0.,4.);
\draw [line width=.8pt] (-1.,3.)-- (0.,3.);
\begin{scriptsize}
\draw [fill=black] (.3,3.) circle (.6pt);
\draw [fill=black] (.5,3.) circle (.6pt);
\draw [fill=black] (.7,3.) circle (.6pt);
\draw [fill=black] (.3,4.) circle (.6pt);
\draw [fill=black] (.5,4.) circle (.6pt);
\draw [fill=black] (.7,4.) circle (.6pt);

\draw [fill=black] (-2.,4.) circle (2.pt);
\draw[color=black] (-2,4.37) node {$1$};
\draw [fill=black] (-1.,4.) circle (2.pt);
\draw[color=black] (-1,4.37) node {$3$};
\draw [fill=black] (-1.,3.) circle (2.pt);
\draw[color=black] (-1,2.63) node {$4$};
\draw [fill=black] (-2.,3.) circle (2.pt);
\draw[color=black] (-2,2.63) node {$2$};
\draw [fill=black] (2.,4.) circle (2.pt);
\draw[color=black] (2.,4.37) node {$2n-1$};
\draw [fill=black] (2.,3.) circle (2.pt);
\draw[color=black] (2.,2.63) node {$2n$};
\end{scriptsize}
\end{tikzpicture}
\qquad
\begin{tikzpicture}[line cap=round,line join=round,>=triangle 45,x=1.0cm,y=1.0cm]
\draw [line width=.8pt] (-2.,4.)-- (-1.,4.);
\draw [line width=.8pt] (-2.,4.)-- (-2.,3.);
\draw [line width=.8pt] (-1.,3.)-- (-2.,3.);
\draw [line width=.8pt] (1.,4.)-- (2.,4.);
\draw [line width=.8pt] (2.,4.)-- (2.,3.);
\draw [line width=.8pt] (2.,3.)-- (1.,3.);
\draw [line width=.8pt] (-1.,4.)-- (0.,4.);
\draw [line width=.8pt] (-1.,3.)-- (0.,3.);
\begin{scriptsize}
\draw [fill=black] (.3,3.) circle (.6pt);
\draw [fill=black] (.5,3.) circle (.6pt);
\draw [fill=black] (.7,3.) circle (.6pt);
\draw [fill=black] (.3,4.) circle (.6pt);
\draw [fill=black] (.5,4.) circle (.6pt);
\draw [fill=black] (.7,4.) circle (.6pt);

\draw [fill=black] (-2.,4.) circle (2.pt);
\draw[color=black] (-2,4.37) node {$1$};
\draw [fill=black] (-1.,4.) circle (2.pt);
\draw[color=black] (-1,4.37) node {$3$};
\draw [fill=black] (-1.,3.) circle (2.pt);
\draw[color=black] (-1,2.63) node {$4$};
\draw [fill=black] (-2.,3.) circle (2.pt);
\draw[color=black] (-2,2.63) node {$2$};
\draw [fill=black] (2.,4.) circle (2.pt);
\draw[color=black] (2.,4.37) node {$2n-1$};
\draw [fill=black] (2.,3.) circle (2.pt);
\draw[color=black] (2.,2.63) node {$2n$};
\end{scriptsize}
\end{tikzpicture}
\begin{tikzpicture}[line cap=round,line join=round,>=triangle 45,x=1.0cm,y=1.0cm,scale = 1]
\draw [line width=.8pt] (-3.,4.)-- (-2.,4.);
\draw [line width=.8pt] (-2.,4.)-- (-1.,4.);
\draw [line width=.8pt] (-2.,4.)-- (-2.,3.);
\draw [line width=.8pt] (-3.,4.)-- (-3.,3.);
\draw [line width=.8pt] (-4.,4.)-- (-4.,3.);
\draw [line width=.8pt] (-5.,4.)-- (-5.,3.);
\draw [line width=.8pt] (-1.,3.)-- (-2.,3.);
\draw [line width=.8pt] (-2.,3.)-- (-3.,3.);
\draw [line width=.8pt] (1.,4.)-- (2.,4.);
\draw [line width=.8pt] (2.,4.)-- (2.,3.);
\draw [line width=.8pt] (2.,3.)-- (1.,3.);
\draw [line width=.8pt] (-1.,4.)-- (0.,4.);
\draw [line width=.8pt] (-1.,3.)-- (0.,3.);
\draw [line width=.8pt] (-5.,4.)-- (-4.,4.);
\draw [line width=.8pt] (-5.,3.)-- (-4.,3.);
\begin{scriptsize}
\draw [fill=black] (-3.3,3.) circle (.6pt);
\draw [fill=black] (-3.5,3.) circle (.6pt);
\draw [fill=black] (-3.7,3.) circle (.6pt);
\draw [fill=black] (-3.3,4.) circle (.6pt);
\draw [fill=black] (-3.5,4.) circle (.6pt);
\draw [fill=black] (-3.7,4.) circle (.6pt);

\draw [fill=black] (.3,3.) circle (.6pt);
\draw [fill=black] (.5,3.) circle (.6pt);
\draw [fill=black] (.7,3.) circle (.6pt);
\draw [fill=black] (.3,4.) circle (.6pt);
\draw [fill=black] (.5,4.) circle (.6pt);
\draw [fill=black] (.7,4.) circle (.6pt);

\draw [fill=black] (-3.,4.) circle (2.pt);
\draw[color=black] (-3,4.37) node {$2k-3$};
\draw [fill=black] (-2.,4.) circle (2.pt);
\draw[color=black] (-2,4.37) node {$2k-1$};
\draw [fill=black] (-1.,4.) circle (2.pt);
\draw[color=black] (-1,4.37) node {$2k+1$};
\draw[color=black] (0,4.37) node {};
\draw[color=black] (0,2.63) node {};
\draw [fill=black] (-1.,3.) circle (2.pt);
\draw[color=black] (-1,2.63) node {$2k+2$};
\draw [fill=black] (-2.,3.) circle (2.pt);
\draw[color=black] (-2,2.63) node {$2k$};
\draw [fill=black] (-3.,3.) circle (2.pt);
\draw[color=black] (-3,2.63) node {$2k-2$};
\draw [fill=black] (2.,4.) circle (2.pt);
\draw[color=black] (2.,4.37) node {$2n-1$};
\draw [fill=black] (2.,3.) circle (2.pt);
\draw[color=black] (2.,2.63) node {$2n$};
\draw [fill=black] (-4.,4.) circle (2.pt);
\draw[color=black] (-4,4.37) node {$3$};
\draw [fill=black] (-4.,3.) circle (2.pt);
\draw[color=black] (-4,2.63) node {$4$};
\draw [fill=black] (-5.,4.) circle (2.pt);
\draw[color=black] (-5,4.37) node {$1$};
\draw [fill=black] (-5.,3.) circle (2.pt);
\draw[color=black] (-5,2.63) node {$2$};
\node at (-3.5,5) {(A)};
\node at (0.75,5) {(B)};
\node at (-1.45,2.2) {(C)};
\end{scriptsize}
\end{tikzpicture}
    \caption{An illustration of the recursive technique to determine resistance distances in the ladder graph.  We begin with a path with $2n$ vertices labeled as in panel (A) and add one edge (with weight zero between nodes $1$ and $2$ as illustrated in panel (B).  We then change the wight of added edge from weight zero to weight one and use Theorem~\ref{thm:recur} to determine the resulting resistance distances.  We continue in this manner adding edges and then changing weights until $n-1$ edges are added.  An intermediate step (after $k$ edges are added) is shown in panel (C).}
    \label{fig:recursive_ladder}
\end{figure}
\subsection{Local Rules}
Another technique used for the determination of resistance distance is the use of local sum rules.  These local rules were developed based off the work of Klein~\cite{klein2002resistance} who developed a set of resistance distance sum rules.  These rules work well when the graph has symmetries or is distance regular, as the inherent symmetries in the graph must be used in order to efficiently solve the resulting system of equations.  More precisely, the local sum rule is as follows:
\begin{theorem}[\cite{chen2008resistance}]
     Let $G$ be a connected graph with unit edge resistance. The relations
     \begin{equation}\label{eq:localsum}\deg(u)r_G(u,v) + \sum_{z\in N(u)}r_G(u,z)-r_G(v,z) = 2;\quad\forall u,v \in V,\end{equation}
     determine all $r_G(i,j)$.
\end{theorem}
This result has been extended to graphs with arbitrary edge resistances~\cite{chen2}.
\begin{examplei}\label{ex:localladder}We return to the familiar example of the ladder graph, and consider the graph with 6 vertices, labeled as shown in Figure~\ref{fig:local_ladder}.  Due to symmetries in this graph, we observe that there are 6 unique resistance distances:
\begin{gather*}r(A,B) = r(B,C)=r(D,E)=r(E,F)\\
r(A,F) = r(C,D)\\
r(A,C)=r(D,F)\\
r(A,D)=r(C,F)\\
r(A,E)=r(B,D)=r(B,F)=r(C,E)\\
r(B,E).\end{gather*}  Applying Equation~\ref{eq:localsum} we find
\[2r(A,B) +r(A,B) -r(B,B)+r(A,F) -r(B,F) = 2,\]
which simplifies to 
\[3r(A,B) + r(A,F) -r(A,E) =2\]
proceeding in a similar manner for $r(A,C),\,r(A,D),\,r(A,E),\,r(A,F)$ and $r(B,E)$, yields the following linear equation
\[\begin{bmatrix}
3 & 0 & 0 & -1 & 1 & 0\\
0 & 2 & -1 & 0 & 1 & 0\\
1 & -1 & 2 & -1 & 1 & 0\\
0 & 0 & 0 & 2 & -1 & 1\\
1 & 0 & 0 & -1 & 3 & 0\\
2 & 0 & 0 & -2 & 0 & 4\\
\end{bmatrix}
\begin{bmatrix}
r(A,B)\\r(A,C)\\r(A,D)\\r(A,E)\\r(A,F)\\r(B,E)
\end{bmatrix}=
\begin{bmatrix}
2\\2\\2\\2\\2\\2
\end{bmatrix}.\]
Solving we find $r(A,B)=11/15$, $r(A,C)=20/15$, $r(A,D)=21/15$, $r(A,E)=14/15$, $r(A,F)=11/15$, and $r(B,E)=9/15$ as expected. We note that for this example, the number of unknowns is equal to the number of vertices, hence solving the system, and solving for the pseudoinverse of the combinatorial Laplacian is approximately equivalent.  For the ladder with $2n$ nodes the number of ``unique'' resistances is $(n^2-n)/2$.  
\end{examplei}
\vspace*{3mm}

\begin{figure}
     \begin{center}
         
\begin{tikzpicture}[line cap=round,line join=round,>=triangle 45,x=1.0cm,y=1.0cm,scale = 1.1]
\draw [line width=.8pt] (3.,4.)-- (3.,3.);
\draw [line width=.8pt] (3.,4.)-- (4.,4.);
\draw [line width=.8pt] (3.,3.)-- (4.,3.);
\draw [line width=.8pt] (4.,4.)-- (4.,3.);
\draw [line width=.8pt] (4.,4.)-- (5.,4.);
\draw [line width=.8pt] (4.,3.)-- (5.,3.);
\draw [line width=.8pt] (5.,4.)-- (5.,3.);
\begin{small}
\draw [fill=black] (3.,4.) circle (1.6pt);
\draw[color=black] (3,4.3) node {$F$};
\draw [fill=black] (3.,3.) circle (1.6pt);
\draw[color=black] (3,2.7) node {$A$};
\draw [fill=black] (4.,4.) circle (1.6pt);
\draw[color=black] (4,4.3) node {$E$};
\draw [fill=black] (4.,3.) circle (1.6pt);
\draw[color=black] (4,2.7) node {$B$};
\draw [fill=black] (5.,4.) circle (1.6pt);
\draw[color=black] (5,4.3) node {$D$};
\draw [fill=black] (5.,3.) circle (1.6pt);
\draw[color=black] (5,2.7) node {$C$};
\end{small}
\begin{scriptsize}

\end{scriptsize}
\end{tikzpicture}
     \end{center}
   \caption{ A ladder graph with $6$ vertices, labeled for Examples~\ref{ex:localladder} and~\ref{ex:simplex}.}
        \label{fig:local_ladder}
 \end{figure}
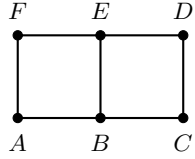

For an arbitrary graph, with no easily discernible symmetries, the number of unique resistances is on the order of $n^2$. As $n$ increases this technique rapidly surpasses the use of the combinatorial Laplacian in complexity.  This technique is useful in the case of highly symmetric graphs, such as regular polyhedra~\cite{chen2008resistance}.

 \subsection{The combinatorial Laplacian}
A different technique that works for graphs whose combinatorial Laplacian have specific, nice structures utilizes the generalized inverse of this matrix.  More specifically, given a graph $G$ with combinatorial Laplacian $L(G)$ the effective resistance between nodes $i$ and $j$ is given by 
\[
r_G(i,j) = (\mathbf{e}_i - \mathbf{e}_j)^T X (\mathbf{e}_i - \mathbf{e}_j),
\]
where $X$ is any generalized inverse of $L(G)$.
Usually $X$ is taken to be $L^\dagger$, the Moore-Penrose inverse of $L(G)$, so we have 
\begin{equation}\label{eq:Laplacian_inverse}
r_G(i,j) = (\mathbf{e}_i - \mathbf{e}_j)^T L^\dagger (\mathbf{e}_i - \mathbf{e}_j).
\end{equation}

Instead of the Moore-Penrose inverse we can consider the generalized inverse (suggested by Bapat~\cite{BapatWheels})
\[H = \begin{bmatrix} L(n|n)^{-1}& 0\\0&0\end{bmatrix},\]
where $L(n|n)$ is the Laplacian matrix with the $n$th row and $n$th column removed.  We observe that for the cycle graph, the path graph, and the fan graph removing the last row and the column yields a tridiagonal matrix.  Fortunately, for graphs that are tridiagonal, a closed, recursive formula exists for the entries in the inverse matrix.  In particular suppose we consider the symmetric tridiagonal matrix
\[T = \begin{bmatrix}
a_1 & b_1 \\
b_1 & a_2 & b_2 \\
& b_2 & \ddots & \ddots \\
& & \ddots & \ddots & b_{n-1} \\
& & & b_{n-1} & a_n
\end{bmatrix}\]
then 
\[T^{-1}_{ij} = \begin{cases}
(-1)^{|j-i|}b_i \cdots b_{j-1} \theta_{i-1} \phi_{j+1}/\theta_n & \text{ if } i \neq j\\
\theta_{i-1} \phi_{j+1}/\theta_n & \text{ if } i = j.\\
\end{cases}\]

Here $\theta_i$ satisfies the recurrence relation $\theta_i = a_i \theta_{i-1} - b_{i-1}^2\theta_{i-2}$ for  $i=2,3,\ldots,n$
with $\theta_0=1$ and $\theta_1 = a_1$, whereas $\phi_i$ satisfies the relation $\phi_i = a_i \phi_{i+1} - b_i^2 \phi_{i+2}$  for $i=n-1,\ldots,1$ with $\phi_{n+1}=1$ and $\phi_{n} = a_n$\cite{DAFONSECA2007283}.  These equations allow us to immediately write the following equation for the resistance distance between nodes $i$ and $j$
\[r_{ij} = \frac{\theta_{i-1} \phi_{i+1}+\theta_{j-1} \phi_{j+1} - 2(-1)^{|j-i|}b_i \cdots b_{j-1} \theta_{i-1} \phi_{j+1}}{\theta_n}.\]
The use of such a recursive formula, along with the equations for generalized Fibonacci numbers allowed Bapat and Gupta to give an explicit formula for the resistance distance between any two vertices in a wheel or fan graph~\cite{BapatWheels}. Here the {\it wheel graph} $W_n$ is defined as the graph on $n+1$ vertices, formed by taking an $n$ cycle and connecting each vertex $1, \ldots, n$ in the cycle to a ``hub'', vertex $n+1$. 


\begin{proposition}[\cite{BapatWheels}]
Let $W_n$ be the wheel graph with unit edge resistances on $n+1$ vertices for $n\geq 3$ with vertex ordering as described above. Then, for $i,j = 1, \ldots, n$
\[
r_{W_n}(i,n+1) = \frac{F_{2n}^2}{F_{4n}-2F_{2n}}, \quad \text{and} \quad
r_{W_n}(i,j) = \frac{F_{2n}^2}{F_{4n}-2F_{2n}}\left(2-\frac{F_{4k}}{F_{2k}}\right) + F_{2k},
\]
where $k$ is the distance between $i$ and $j$ in $C_n$, and $F_k$ is the $k$th Fibonacci number.
\end{proposition}

\begin{examplei}
We consider the slightly more complex case of the straight linear 2-tree and demonstrate how this technique can be used to find the resistance distance between node $n-1$ and node $n$.  This corresponds to finding the determinant of two different matrices and dividing one by the other.  The numerator corresponds to the determinant of the $(n-2) \times (n-2)$ matrix created by removing the final two rows and columns from the Laplacian.  The denominator is the determinant of the $(n-1)\times(n-1)$ matrix created by removing the final row and column from the determinant.  We note that the value of the denominator is the number of spanning trees of the straight linear 2--tree and is known to be $F_{2n-2}$, so we omit this calculation.

Thus we wish to find the determinant of $M_n$ which in the $(n-2)\times (n-2)$ pentadiagonal matrix with entries
\[(M_n)_{ij} =\begin{cases}
2 & \text{if $i = j =1$}\\
3 & \text{if $i = j =2$}\\
4 & \text{if $i=j$ and $3\leq i \leq n-2$}\\
-1 & \text{if $ |i-j| = 1$ or $|i-j| = 2.$}
\end{cases}\]
We will use Sweet's formula~\cite{sweet}
\begin{multline*}d_n= \left(a_n - \frac{c_{n-2}}{b_{n-2}}\right)d_{n-1} - \left(b_{n-1}-\frac{a_{n-1}c_{n-2}}{b_{n-2}}\right) d_{n-2} -(\beta_{n-2}a_{n-1}-c_{n-2})d_{n-3}\\+\beta_{n-3}\left(\beta_{n-2} - \frac{a_{n-2}c_{n-2}}{b_{n-2}}\right)d_{n-4} + \frac{\beta_{n-3}\beta_{n-4}c_{n-2}}{b_{n-2}}d_{n-5},\end{multline*}
here $d_n$ is the determinant of $M_n$, $d_{n-1}$ is the determinant of $M_n(n-2|n-2)$, $d_{n-2}$ is the determinant of $M_n$ with the last two rows and columns removed. Moreover $a_i = (M_n)_{ii}$, $b_i = (M_n)_{i,i+1}(M_n)_{i+1,i}$, $\beta_i = (M_n)_{i,i+2}(M_n)_{i+2,i}$ and $c_i = (M_n)_{i,i+1}(M_n)_{i+1,i+2}(M_n)_{i+2,i}$.  If the sub and superdiagonals of the matrix in question are all equal to $-1$ as in our case Sweet's formula simplifies to:
\[d_n = (a_n+1)d_{n-1}-(1+a_{n-1})(d_{n-2}+d_{n-3}) + (1+a_{n-2})d_{n-4}-d_{n-5}.\]
\begin{lemma}\label{lem:deteasy} The determinant of $M_n$ is equal to $F_{2n-3}$.
\end{lemma}
We proceed by induction.  For $n=5, 6, 7,8, 9$, straight forward numerical calculations show the determinants of $M_5$, $M_6$, $M_7$, $M_8$, and $M_9$ to be 13, 34, 89, 233, and 610 respectively, thus showing the necessary base cases.  We now assume that $\det(M_n) = F_{2n-3}$ and will show, using the simplified version of Sweet's formula that $\det(M_{n+1}) = F_{2n-1}$.
\begin{align*}
\det(M_{n+1}) &= (4+1)\det(M_n)-(4+1)(\det(M_{n-1})+\det(M_{n-2}))\\&\quad\quad+(4+1)\det(M_{n-3})-\det(M_{n-4})\\
&=5F_{2n-3}-5(F_{2n-5}+F_{2n-7}) + 5F_{2n-9}-F_{2n-11}\\
&=5F_{2n-3}-5(F_{2n-5}+F_{2n-7}) + 2F_{2n-9}+F_{2n-7}\\
&=5F_{2n-3}-5F_{2n-5}- 2F_{2n-6}\\
&=5F_{2n-3}-5F_{2n-5}- 2F_{2n-6}\\
&=3F_{2n-3}-F_{2n-5}\\
&=F_{2n-1}
\end{align*}

From this Lemma, and the knowledge that the number of spanning trees in a straight linear 2--tree with $n$ vertices is equal to $F_{2n-2}$, we can conclude that the $r(n-1,n) = F_{2n-3}/F_{2n-2}$ which is consistent with the results in Theorem~\ref{thm:sl2t}.
\end{examplei}

\begin{examplei}
In the case of the straight linear 2--tree and the ladder graphs the resulting $L(n|n)$ is a non-singular pentadiagonal matrix.  If the pentadiagonal matrix is Toeplitz a recursive formula for the entries in the inverse matrix is known~\cite{WANG201512}, but no such general formula exists for an arbitrary pentadiagonal matrix, only numerical methods that allow for fast computation~\cite{Zhao2008OnTI}.  

If instead we consider the simpler problem of determining the resistance distance between the extremal vertices (i.e., $r(1,2n)$ in the case of the ladder graph), the problem reduces to determining the entry $ H(1,1)$ (because $H(1,2n)=H(2n,1)=H(2n,2n)=0)$).  Thus only a single entry in the matrix needs to be determined, and that entry is  just the determinant of L with the first row and column, and the last row and column removed, divided by the number of spanning trees (which is known).  

Unfortunately, even though a seven-term recursive formula exists to calculate the determinant of a pentadiagonal matrix~\cite{sweet}, this formula requires all the entries on the sub and super diagonals to be non-zero, which is not the case for the ladder graph.  This formula does offer an alternative way, however, of calculating the resistance distance $r(1,n)$ in the straight linear 2--tree, similar to the prior example.

  \end{examplei}

\subsection{As a solution to an optimization problem}
In addition to combinatorial and matrix based techniques, resistance distance can be framed as an optimization problem.  This is unsurprising as electrical circuits can be modeled as springs and finding the resistance distance in essence is equivalent to finding the steady state solution of a spring network.  The following theorem formalizes the resistance distance between two vertices as a problem of finding the minimum.\begin{theorem}\label{lem:res_min}
\[\frac{1}{r(u,v)} = \min_{\substack{x\in\mathbb R^n\\x_u-x_v=1}}\sum_{\{ij\}\in E(G)}(x_i-x_j)^2.\]
\end{theorem}

\begin{examplei}
We consider the ladder graph with six vertices labeled as in Figure~\ref{fig:local_ladder}, and consider the resistance distance between nodes $A$ and $F$.  Using Theorem~\ref{lem:res_min}, we know
\begin{align*}
\frac{1}{r(A,F)} &= \min_{\substack{x\in\mathbb R^n\\x_A-x_F=1}}\sum_{\{ij\}\in E(G)}(x_i-x_j)^2\\
&=\min (x_A-x_B)^2+(x_A-x_F)^2+(x_B-x_C)^2+(x_B-x_E)^2 \\
& \hspace*{5cm}+(x_C-x_D)^2+(x_D-x_E)^2+(x_E-x_F)^2\\
&= \min (x_F+1-x_B)^2+1+(x_B-x_C)^2+(x_B-x_E)^2\\
& \hspace*{5cm}+ (x_C-x_D)^2+(x_D-x_E)^2+(x_E-x_F)^2
\end{align*}
The minimum occurs when the gradient is equal to zero, or in other words when
\[\begin{bmatrix} 3 & -1 & 0 & -1 & -1\\
-1 & 2 & -1 & 0 & 0\\
0 & -1 & 2 & -1 & 0\\
-1 & 0 & -1 & 3 & -1\\
-1 & 0 & 0 & -1 & 2\\
\end{bmatrix}\begin{bmatrix} x_B\\x_C\\x_D\\x_E\\x_F\end{bmatrix}=\begin{bmatrix} 1 \\ 0 \\ 0\\ 0\\ -1\end{bmatrix}.\]
A moments consideration will reveal that this is not a linearly independent system of equations, thus there is no unique solution.  However, algebra yields 
$    x_A = x_F+1,\;
    x_B=  x_F+ \frac{7}{11},\;
    x_C=  x_F+ \frac{6}{11},\;
    x_D=  x_F+ \frac{5}{11},$ and 
    $x_E=  x_F+ \frac{4}{11}.$
Thus, setting $x_F = 0$ gives values for $x_A, x_B, x_C, x_D$ and $x_E$. Checking, 
\begin{align*}
r(A,F)\! &=\! ((1\!-\textstyle\frac{7}{11})^2+1+(\frac{7}{11}-\frac{6}{11})^2 + (\frac{7}{11}-\frac{4}{11})^2 + (\frac{6}{11}-\frac{5}{11})^2+(\frac{5}{11}-\frac{4}{11})^2+(\frac{4}{11})^2)^{-1}\\& = (\textstyle\frac{15}{11})^{-1} = \frac{11}{15},
\end{align*}
as expected.

\end{examplei}

\subsection{Simplices and distances}
A less well-known method for determining resistance distances in a graph lies in the construction of certain simplices.  Recall that a geometric simplex is a generalization of a triangle to a higher or lower dimension, and a simplex in dimension $n$ is defined by $n+1$ linearly independent points.  We represent this in matrix form as $S$ which is $n \times (n+1)$.  To determine resistance distance we will be interested in \emph{hyperacute simplices}, that is to say those simplices whose dihedral angles are non-obtuse.  As a first step to the relationship between resistance distance and simplicies we state the following result from Fiedler~\cite{fiedlerLapl}
\begin{theorem}
     There is a bijection between
     \begin{enumerate}
         \item Laplacian matrices of $n$ dimensions,
         \item hyperacute simplices on $n$ vertices.
     \end{enumerate}
\end{theorem}
With the knowledge that such a bijection exists, it should not be surprising, given the relationship between the pseudoinverse of the combinatorial Laplacian and the resistance distance that a similar relationship exists between these simplices and the resistance distance.  Before stating that relationship, we first give the necessary details involved in constructing this bijection.    
The Laplacian should be dependent on the relative locations of the vertices of the simplex, but should not be dependent on the location of the simplex in space, or rotations of the simplex.  For this reason we introduce a centered simplex $\tilde{S} = S(I-\mathbf{u}\mathbf{u}^T/n)$, where $u$ is a vector of ones, which is an identical simplex, except that the centroid of the translated simplex is located at the origin.  
With this translation, the corresponding Laplacian is the Moore-Penrose pseudoinverse of the Gram matrix of the translated simplex, i.e., $L=((I-\mathbf{u}\mathbf{u}^T/n)S^TS(I-\mathbf{u}\mathbf{u}^T/n))^{\dagger}=(\tilde{S}^T\tilde{S})^\dagger$ for any representative simplex $S$.  Given a Laplacian, with spectral decomposition $L=\sum_{k=1}^{n-1}\mu_k \mathbf{z}_k \mathbf{z}_k^T$, we can construct the corresponding simplex as an equivalence class of the simplex with vertices 
\begin{equation}\label{eq:get_locaation}(s_i)_k= (z_k)_ i\sqrt{1/\mu_k}.\end{equation}  

With this bijection formally stated, recall Equation~\ref{eq:Laplacian_inverse} which relates the resistance distance to entries in $L^\dagger$.  More specifically we observe
\[r(i,j) =(\mathbf{e}_i - \mathbf{e}_j)^T L^\dagger (\mathbf{e}_i - \mathbf{e}_j)=(\mathbf{e}_i - \mathbf{e}_j)^T \tilde{S}^T\tilde{S} (\mathbf{e}_i - \mathbf{e}_j)
=||\tilde{S}(\mathbf{e}_i - \mathbf{e}_j)||^2_2,\]
or in other words the resistance distance between two points in a network is equal to the square of the (Euclidean) distance between the points in the centered Simplex.  This result can be summarized in the following theorem known as, Fiedler's identity, which gives the relationship between between a simplex, and the associated resistance distance matrix $\Omega$.

\begin{theorem}~\cite{devriendt2020effective}
     For a weighted graph $G$ with Laplacian matrix $L$ and simplex $S$ with resistance distance matrix $\Omega$, the following identity holds
     \begin{equation}
         -\frac{1}{2}\begin{bmatrix}0& u^T\\ u &\Omega\end{bmatrix}
         = \begin{bmatrix}4R^2&-2r^T\\-2r&  L\end{bmatrix}^{-1}
         \end{equation}
         where $u$ is a vector of ones, $r=\frac{1}{2}L\zeta+u/n$  with $\zeta = \text{diag}(L^{\dagger})$ determines the circumcenter of $S$ as $Sr$, and with $R=\sqrt{\frac{1}{2}\zeta(r+u/n)}$ the circumradius of $S$.

\end{theorem}
\begin{examplei}\label{ex:simplex}
We consider the ladder graph with six vertices labeled as in Figure~\ref{fig:local_ladder}.  Upon calculating the spectral decomposition of the associated Laplacian and using Equation~\ref{eq:get_locaation} we construct the 5-dimensional simplex with vertices located at
\begin{align*} 
s_A&= \begin{bmatrix} -1/2 & -1/\sqrt{12} & -1/6 & 1/\sqrt{12} & 1/\sqrt{60}\end{bmatrix}\\
s_B&= \begin{bmatrix} 0 & -1/\sqrt{12} & -1/6 & -1/\sqrt{12} & -1/\sqrt{15}\end{bmatrix}\\
s_C&= \begin{bmatrix} 1/2 & -1/\sqrt{12} & 1/3 & 0 & 1/\sqrt{60}\end{bmatrix}\\
s_D&= \begin{bmatrix} 1/2 & 1/\sqrt{12} & -1/6 & 1/\sqrt{12} & -1/\sqrt{60}\end{bmatrix}\\
s_E&= \begin{bmatrix} 0 & 1/\sqrt{12} & -1/6 & -1/\sqrt{12} & 1/\sqrt{15}\end{bmatrix}\\
s_F&= \begin{bmatrix} -1/2 & 1/\sqrt{12} & 1/3 & 0 & -1/\sqrt{60}\end{bmatrix}.
\end{align*}
As a check we calculate 
\begin{align*}
r(A,B) &= ||s_A-s_B||^2_2\\ &= ||\begin{bmatrix}-1/2&0&0 &2/\sqrt{12}&3/\sqrt{60} \end{bmatrix}||^2_2 \\ &=1/4+1/3+3/20=11/15,
\end{align*}
as expected.
\end{examplei}

\subsection{Numerical techniques for estimating the resistance distance}
In some applications of resistance distance, the actual numerical value of the resistance distance is less important than knowing the approximate magnitude of the resistance distance or the ordering of resistance distance between nodes.  In this subsection we briefly review two numerical techniques for estimating resistance distance in a (typically) very large graph.  Although these techniques do not return exact answers, the answers they return are often sufficient for applications.
\subsubsection{Numerically estimating the group inverse}
In~\cite{Pachev} the authors use spectral embedding to create estimates for resistance distance.  More specifically, given the spectral decomposition of the Laplacian matrix $L=\sum_{k=1}^{n-1}\mu_k \mathbf{z}_k \mathbf{z}_k^T$, the spectral decomposition of the Moore-Penrose pseudo inverse is given by $L^{\dagger}=\sum_{k=1}^{n-1}\frac{1}{\mu_k} \mathbf{z}_k \mathbf{z}_k^T$.  Hence the resistance distance betweeen nodes $i$ and $j$ is given by \[r(i,j)= \sum_{k=1}^{n-1}\frac{(z_{ki} -z_{kj})^2}{\lambda_k}.\]  This suggests that a reasonable estimate of the resistance distance can be obtained by considering the $t$ smallest eigenvalues and eigenvectors associated with them, in essence 
\[r(i,j)\approx \sum_{k=1}^{t}\frac{(z_{ki} -z_{kj})^2}{\lambda_k}.\]
Moreover, the $t$ eigenvalues and eigenvectors are approximated using iterative methods from numerical linear algebra (e.g., Arnoldi method).  Pachev and Webb found that using this technique to approximate resistance distance for the problem of link prediction gave comparable accuracy to exact techniques for large networks but in a fraction of the time required for exact techniques.
\subsection{Random Projections}
In~\cite{SpielSparse} Spielman and Srivastava developed an algorithm for computing an approximate resistance distance in $O(\log n)$ time using random projections.  Key to their technique is the knowledge that resistance distance is a just the distance between vectors in higher dimensional space.  They then using a random linear projection they project these vectors into a lower dimensional space in a numerically efficient manner.  They take advantage of the Johnson-Lindenstrauss, which states that with high probability these types of random projections preserve distances.  Working in this lower dimensional space, they then use a previously developed solver (\texttt{STSolve}) to quickly compute the distances.

\subsection{Additional techniques}
In addition to the algorithms and techniques discussed in this survey a wide variety of techniques have also been developed and studied to determine resistance distances. For completeness we include the following list:
\begin{itemize}

    \item The commute time and escape probability of random walks~\cite{doylesnell}.
    \item Eigenvalues and eigenvectors of the Laplacian matrix~\cite{klein1997graph} and the normalized Laplacian matrix~\cite{chenzhang}.
    \item Minors of the Laplacian matrix~\cite{BAPAT20111479}.
    \end{itemize}


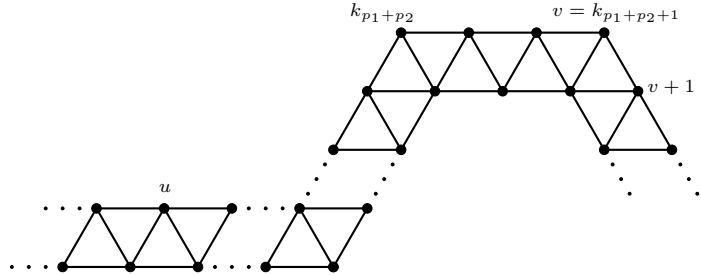
\begin{figure}
\begin{center}
\begin{tikzpicture}[line cap=round,line join=round,>=triangle 45,x=1.0cm,y=1.0cm, scale = .9]
\draw [line width=.8pt] (-0.5,0.8660254037844387)-- (0.5,0.8660254037844384);
\draw [line width=.8pt] (0.5,0.8660254037844384)-- (0.,0.);
\draw [line width=.8pt] (0.,0.)-- (-0.5,0.8660254037844387);
\draw [line width=.8pt] (0.,0.)-- (1.,0.);
\draw [line width=.8pt] (1.,0.)-- (0.5,0.8660254037844384);
\draw [line width=.8pt] (0.5,0.8660254037844384)-- (1.5,0.8660254037844376);
\draw [line width=.8pt] (1.5,0.8660254037844376)-- (1.,0.);
\draw [line width=.8pt] (1.,0.)-- (2.,0.);
\draw [line width=.8pt] (2.,0.)-- (1.5,0.8660254037844376);
\draw [line width=.8pt] (1.5,0.8660254037844376)-- (2.5,0.8660254037844367);
\draw [line width=.8pt] (2.5,0.8660254037844367)-- (2.,0.);
\draw [line width=.8pt] (2.5,0.8660254037844367)-- (3.,0.);
\draw [line width=.8pt] (3.,0.)-- (2.,0.);
\draw [line width=.8pt] (2.,0.)-- (2.5,-0.8660254037844423);
\draw [line width=.8pt] (2.5,-0.8660254037844423)-- (3.,0.);
\draw [line width=.8pt] (3.,0.)-- (3.5,-0.8660254037844444);
\draw [line width=.8pt] (3.5,-0.8660254037844444)-- (2.5,-0.8660254037844423);
\draw [fill=black] (2.875,-1.5155) circle (.6pt);
\draw [fill=black] (2.75,-1.299) circle (.6pt);
\draw [fill=black] (2.625,-1.0825) circle (.6pt);
\draw [fill=black] (3.875,-1.5155) circle (.6pt);
\draw [fill=black] (3.75,-1.299) circle (.6pt);
\draw [fill=black] (3.625,-1.0825) circle (.6pt);
\draw [line width=.8pt] (-0.5,0.8660254037844387)-- (-1.,0.);
\draw [line width=.8pt] (0.,0.)-- (-1.,0.);
\draw [line width=.8pt] (0.,0.)-- (-0.5,-0.8660254037844388);
\draw [line width=.8pt] (-0.5,-0.8660254037844388)-- (-1.,0.);
\draw [line width=.8pt] (-1.,0.)-- (-1.5,-0.8660254037844382);
\draw [line width=.8pt] (-1.5,-0.8660254037844382)-- (-0.5,-0.8660254037844388);
\draw [fill=black] (-.875,-1.5155) circle (.6pt);
\draw [fill=black] (-.75,-1.299) circle (.6pt);
\draw [fill=black] (-.625,-1.0825) circle (.6pt);
\draw [fill=black] (-1.875,-1.5155) circle (.6pt);
\draw [fill=black] (-1.75,-1.299) circle (.6pt);
\draw [fill=black] (-1.625,-1.0825) circle (.6pt);
\draw [line width=.8pt] (-2.,-1.7320508075688767)-- (-1.,-1.7320508075688776);
\draw [line width=.8pt] (-1.,-1.7320508075688776)-- (-1.5,-2.598076211353316);
\draw [line width=.8pt] (-1.5,-2.598076211353316)-- (-2.,-1.7320508075688767);
\draw [line width=.8pt] (-1.5,-2.598076211353316)-- (-2.5,-2.5980762113533147);
\draw [line width=.8pt] (-2.5,-2.5980762113533147)-- (-2.,-1.7320508075688767);
\draw [fill=black] (-2.25,-1.732) circle (.6pt);
\draw [fill=black] (-2.75,-1.732) circle (.6pt);
\draw [fill=black] (-2.5,-1.732) circle (.6pt);
\draw [fill=black] (-3.25,-2.598) circle (.6pt);
\draw [fill=black] (-2.75,-2.598) circle (.6pt);
\draw [fill=black] (-3,-2.598) circle (.6pt);
\draw [line width=.8pt] (-3.5,-2.598076211353313)-- (-3.,-1.7320508075688743);
\draw [line width=.8pt] (-3.,-1.7320508075688743)-- (-4.,-1.7320508075688725);
\draw [line width=.8pt] (-4.,-1.7320508075688725)-- (-3.5,-2.598076211353313);
\draw [line width=.8pt] (-3.5,-2.598076211353313)-- (-4.5,-2.598076211353311);
\draw [line width=.8pt] (-4.5,-2.598076211353311)-- (-4.,-1.7320508075688725);
\draw [line width=.8pt] (-4.,-1.7320508075688725)-- (-5.,-1.7320508075688708);
\draw [line width=.8pt] (-5.,-1.7320508075688708)-- (-4.5,-2.598076211353311);
\draw [line width=.8pt] (-4.5,-2.598076211353311)-- (-5.5,-2.59807621135331);
\draw [line width=.8pt] (-5.5,-2.59807621135331)-- (-5.,-1.7320508075688708);
\draw [fill=black] (-5.25,-1.732) circle (.6pt);
\draw [fill=black] (-5.75,-1.732) circle (.6pt);
\draw [fill=black] (-5.5,-1.732) circle (.6pt);
\draw [fill=black] (-6.25,-2.598) circle (.6pt);
\draw [fill=black] (-5.75,-2.598) circle (.6pt);
\draw [fill=black] (-6,-2.598) circle (.6pt);
\begin{scriptsize}
\draw [fill=black] (0.,0.) circle (2.pt);
\draw [fill=black] (-0.5,0.8660254037844387) circle (2.pt);
\draw[color=black] (-0.7663617143955417,1.1586699980780315) node {$k_{p_1+p_2}$};
\draw [fill=black] (-1.,0.) circle (2.pt);
\draw [fill=black] (0.5,0.8660254037844384) circle (2.pt);
\draw [fill=black] (1.,0.) circle (2.pt);
\draw [fill=black] (-0.5,-0.8660254037844388) circle (2.pt);
\draw [fill=black] (-1.5,-0.8660254037844382) circle (2.pt);
\draw [fill=black] (2.,0.) circle (2.pt);
\draw [fill=black] (2.5,0.8660254037844367) circle (2.pt);
\draw[color=black] (2.677805112435133,1.1586699980780315) node {$v=k_{p_1+p_2+1}$};
\draw [fill=black] (3.,0.) circle (2.pt);
\draw[color=black] (3.491331923890063,0.06699211993080953) node {$v+1$};
\draw [fill=black] (2.5,-0.8660254037844423) circle (2.pt);
\draw [fill=black] (3.5,-0.8660254037844444) circle (2.pt);
\draw [fill=black] (1.5,0.8660254037844376) circle (2.0pt);
\draw [fill=black] (-1.,-1.7320508075688776) circle (2.pt);
\draw [fill=black] (-2.,-1.7320508075688767) circle (2.pt);
\draw [fill=black] (-1.5,-2.598076211353316) circle (2.pt);
\draw [fill=black] (-2.5,-2.5980762113533147) circle (2.pt);
\draw [fill=black] (-1.5,-2.5980762113533165) circle (2.pt);
\draw [fill=black] (-3.,-1.7320508075688743) circle (2.pt);
\draw [fill=black] (-3.5,-2.598076211353313) circle (2.pt);
\draw [fill=black] (-4.,-1.7320508075688725) circle (2.pt);
\draw[color=black] (-3.987580242167981,-1.4245127810878323) node {$u$};
\draw [fill=black] (-4.5,-2.598076211353311) circle (2.pt);
\draw [fill=black] (-5.,-1.7320508075688708) circle (2.pt);
\draw [fill=black] (-5.5,-2.59807621135331) circle (2.pt);
\end{scriptsize}
\end{tikzpicture}
\end{center}
    \caption{An example graph showing the dilemma we face traveling from $v$ to $v+1$, through a bend.}
    \label{fig:passingabend}
\end{figure}

\section{Conjectures and open questions}\label{sec:openstuff}

We recall the Definition~\ref{def:kpath} of a $K$-path as a $K$-tree with exactly two vertices of degree $K$
	It seems natural to use the methods demonstrated in Section \ref{sec:networktransformations} to find resistances between  nodes in  linear $K$-trees for $K\geq 3$, especially in the straight case. However, it is much more complicated to find an algorithm to ``eliminate a $K_n$'' using equivalent network transformations when $n\geq 3$.  
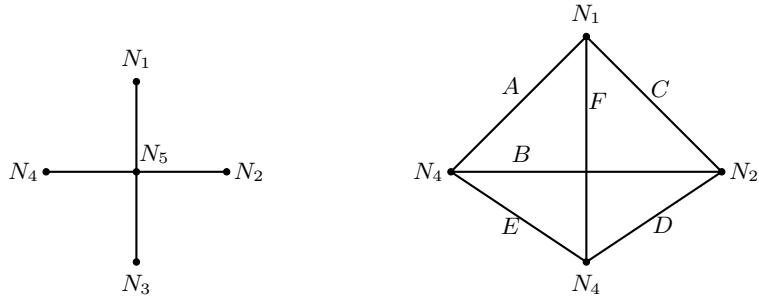
\begin{figure}[h!]
    \centering
\begin{tikzpicture}[line cap=round,line join=round,>=triangle 45,x=1.0cm,y=1.0cm,scale=.6]
\draw [line width=.8pt] (2.,0.)-- (2.,2.);
\draw [line width=.8pt] (2.,4.)-- (2.,2.);
\draw [line width=.8pt] (4.,2.)-- (2.,2.);
\draw [line width=.8pt] (0.,2.)-- (2.,2.);
\begin{small}
\draw [fill=black] (2.,4.) circle (2.pt);
\draw[color=black] (2.,4.5) node {$N_1$};
\draw [fill=black] (2.,0.) circle (2.pt);
\draw[color=black] (2.,-.5) node {$N_3$};
\draw [fill=black] (0.,2.) circle (2.pt);
\draw[color=black] (-.5,2) node {$N_4$};
\draw [fill=black] (4.,2.) circle (2.pt);
\draw[color=black] (4.5,2) node {$N_2$};
\draw [fill=black] (2.,2.) circle (2.pt);
\draw[color=black] (2.4,2.4) node {$N_5$};
\end{small}
\end{tikzpicture}
\hspace*{.6in}
\begin{tikzpicture}[line cap=round,line join=round,>=triangle 45,x=1.0cm,y=1.0cm,scale=.6]
\draw [line width=.8pt] (2.,5.)-- (2.,0.);
\draw [line width=.8pt] (-1.,2.)-- (5.,2.);
\draw [line width=.8pt] (2.,5.)-- (-1.,2.);
\draw [line width=.8pt] (2.,5.)-- (5.,2.);
\draw [line width=.8pt] (5.,2.)-- (2.,0.);
\draw [line width=.8pt] (2.,0.)-- (-1.,2.);
\begin{small}
\draw [fill=black] (2.,5.) circle (2.pt);
\draw [fill=black] (2.,0.) circle (2.pt);
\draw [fill=black] (-1.,2.) circle (2.pt);
\draw [fill=black] (5.,2.) circle (2.pt);
\draw[color=black] (2.25,3.57) node {$F$};
\draw[color=black] (0.56,2.39) node {$B$};
\draw[color=black] (0.32,3.91) node {$A$};
\draw[color=black] (3.64,3.83) node {$C$};
\draw[color=black] (3.7,0.83) node {$D$};
\draw[color=black] (0.32,0.79) node {$E$};
\draw[color=black] (2,5.5) node {$N_1$};
\draw[color=black] (5.5,2) node {$N_2$};
\draw[color=black] (-1.5,2) node {$N_4$};
\draw[color=black] (2,-0.5) node {$N_4$};

\end{small}
\end{tikzpicture}
    \caption{The graph in the left panel is a star circuit with $N=4$ resistors (edges). The graph in the right panel is a mesh circuit with $N = 6$ resistors. }
    \label{fig:mesh4}
\end{figure}

The first question that may arise is: do there exist transformations analogous to $\Delta$--Y and Y--$\Delta$ transformations, for tetrahedra and stars?  The answer is partially afirmative: the star-mesh transformation is defined as
\[
\frac{1}{R_{uv}}={R_u R_v\left(\sum \frac{1}{R}\right)}
\]
where $R_{uv}$ is the resistance on the new edge $uv$ and $R_u$ is the resistance between node $u$ and the center node.  This transformation replaces the $N$ resistors (edges) of the star with $\frac{1}{2}N(N-1)$ resistors (edges) in a mesh, and for $N > 3$ the number of resistors increases.  See Figure~\ref{fig:mesh4}.  Although the star-mesh transformation always exists, because of the decrease in the number of resistors, the mesh--star transformation is only guaranteed to exist for $N \leq 3$ and it is only unique for $N=3$.  With the addition of certain conditions on the resistances on the edges of the mesh, one can guarantee the existence and uniqueness of a mesh--star transformation.  In particular, considering the specialized case of $N = 4$ shown in Figure~\ref{fig:mesh4} the requirement for such a transformation to exist is that 
$AD=CE=BF.$
 Given a 3-tree with unit edge resistances, the above condition holds for any tetrahedron in the network. However, after performing one mesh--star  transformation, it no longer holds for subsequent transformations. Thus, the mesh--star and star--mesh transformations do not produce a straightforward method for finding resistance distances in a $K$-path.

An alternative idea is an algorithm that collapses the graph by taking a sequence of $\Delta$--Y and Y--$\Delta$ transformations.  We give an algorithm to do just this in the case of a straight linear 3-tree. 

\begin{algorithm}\label{alg:3tree}  Given a straight linear 3-tree $G$ with $n$ nodes labeled as suggested in Figure \ref{fig:3tree}. We can perform equivalent network transformations in such a way as to reduce the number of tetrahedra separating the nodes of degree 3 in the following way: 
\begin{enumerate}
    \item Perform a $\Delta$--Y transformation on the cycle involving nodes 1, 2, and 3. This creates a new node, labelled $\star$ in Figure \ref{fig:3tree}. 
    \item Perform a Y--$\Delta$ transformation on the star with center at node 2. This eliminates node 2. 
    \item Perform a $\Delta$--Y transformation on the cycle involving nodes 1, $\star$, and 4. This creates a new node, labelled $\ast$ in Figure \ref{fig:3tree}.
    \item Perform a Y--$\Delta$ transformation on the star with center at node $\star$. This eliminates node $\star$. 
\end{enumerate}
We now have one fewer tetrahedra separating node $1$ from node $n$. Note that the node set is not stable under this algorithm, as we have eliminated node 2 and introduced a new node $\ast$. 
\end{algorithm}

Although Algorithm \ref{alg:3tree} can be used to find  resistance distances between any pair of vertices on the straight linear 3-trees  with $n$ vertices for any $n$, we do not know of any general formulae for these resistances.

\begin{figure}[ht!]
    \centering
\begin{tikzpicture}[line cap=round,line join=round,>=triangle 45,x=1.0cm,y=1.0cm, scale = .8]
\draw [line width=.8pt] (0.7119648044207935,0.056502145327751115)-- (-0.19048706476012342,2.0870188509848124);
\draw [line width=.8pt] (-0.19048706476012342,2.0870188509848124)-- (2.6636787078204596,1.3307124522628444);
\draw [line width=.8pt] (2.6636787078204596,1.3307124522628444)-- (0.7119648044207935,0.056502145327751115);
\draw [line width=.8pt] (0.7119648044207935,0.056502145327751115)-- (1.2007929002271236,2.481841543751463);
\draw [line width=.8pt] (1.2007929002271236,2.481841543751463)-- (-0.19048706476012342,2.0870188509848124);
\draw [line width=.8pt] (1.2007929002271236,2.481841543751463)-- (2.6636787078204596,1.3307124522628444);
\draw [line width=.8pt] (1.5392123511699674,3.685110702659351)-- (-0.19048706476012342,2.0870188509848124);
\draw [line width=.8pt] (1.5392123511699674,3.685110702659351)-- (1.2007929002271236,2.481841543751463);
\draw [line width=.8pt] (1.5392123511699674,3.685110702659351)-- (2.6636787078204596,1.3307124522628444);
\draw [line width=.8pt] (1.5392123511699674,3.685110702659351)-- (3.43812149257148,3.478298815972058);
\draw [line width=.8pt] (3.43812149257148,3.478298815972058)-- (2.6636787078204596,1.3307124522628444);
\draw [line width=.8pt] (1.2007929002271236,2.481841543751463)-- (3.43812149257148,3.478298815972058);
\begin{scriptsize}
\draw [fill=black] (0.7119648044207935,0.056502145327751115) circle (2.pt);
\draw[color=black] (0.6743626432049219,-0.19731244287938157) node {1};
\draw [fill=black] (-0.19048706476012342,2.0870188509848124) circle (2.pt);
\draw[color=black] (-0.547707596310903,2.171623713720523) node {2};
\draw [fill=black] (2.6636787078204596,1.3307124522628444) circle (2.pt);
\draw[color=black] (2.836486913117535,1.2691718445396072) node {$3$};
\draw [fill=black] (1.2007929002271236,2.481841543751463) circle (2.pt);
\draw[color=black] (0.9563788523239585,2.773258293174467) node {$4$};
\draw [fill=black] (1.5392123511699674,3.685110702659351) circle (2.pt);
\draw[color=black] (1.3700026256985456,3.995328532690291) node {$5$};
\draw [fill=black] (3.43812149257148,3.478298815972058) circle (2.pt);
\draw[color=black] (3.6449333792587733,3.6757101623553834) node {$6$};
\draw [fill=black] (2.6636787078204605,1.3307124522628435) circle (2.0pt);
\end{scriptsize}
\draw node at (1.5,-1.3) {Initial 3-tree};

\end{tikzpicture}
\hspace*{.4cm}
%
\begin{tikzpicture}[line cap=round,line join=round,>=triangle 45,x=1.0cm,y=1.0cm, scale = .8]
\draw [line width=.8pt] (1.8588307215048754,-0.24431514439922092)-- (1.501610189954096,2.256228576456234);
\draw [line width=.8pt] (1.501610189954096,2.256228576456234)-- (-0.021277339288701486,2.4630404631435274);
\draw [line width=.8pt] (1.501610189954096,2.256228576456234)-- (2.874089074333407,1.4477821103149968);
\draw [line width=.8pt] (1.501610189954096,3.6663096220514153)-- (-0.021277339288701486,2.4630404631435274);
\draw [line width=.8pt] (1.501610189954096,3.6663096220514153)-- (1.501610189954096,2.256228576456234);
\draw [line width=.8pt] (1.501610189954096,3.6663096220514153)-- (2.874089074333407,1.4477821103149968);
\draw [line width=.8pt] (1.501610189954096,3.6663096220514153)-- (3.430207224692275,3.4147118575080917);
\draw [line width=.8pt] (3.430207224692275,3.4147118575080917)-- (2.874089074333407,1.4477821103149968);
\draw [line width=.8pt] (1.501610189954096,2.256228576456234)-- (3.430207224692275,3.4147118575080917);
\draw [line width=.8pt] (-0.021277339288701486,2.4630404631435274)-- (0.4111475146938212,1.0341583369404102);
\draw [line width=.8pt] (0.4111475146938212,1.0341583369404102)-- (2.874089074333407,1.4477821103149968);
\draw [line width=.8pt] (0.4111475146938212,1.0341583369404102)-- (1.8588307215048754,-0.24431514439922092);
\begin{scriptsize}
\draw [fill=black] (1.8588307215048754,-0.24431514439922092) circle (2.pt);
\draw[color=black] (1.821228560289004,-0.4981297326063535) node {1};
\draw [fill=black] (-0.021277339288701486,2.4630404631435274) circle (2.pt);
\draw[color=black] (-0.3784978708394811,2.5476453258792384) node {2};
\draw [fill=black] (2.874089074333407,1.4477821103149968) circle (2.pt);
\draw[color=black] (3.0620998804127644,1.4007794087951573) node {$3$};
\draw [fill=black] (1.501610189954096,2.256228576456234) circle (2.pt);
\draw[color=black] (1.2759972226588667,2.0588172300729086) node {$4$};
\draw [fill=black] (1.501610189954096,3.6663096220514153) circle (2.pt);
\draw[color=black] (1.3888037063064813,4.051731774514098) node {$5$};
\draw [fill=black] (3.430207224692275,3.4147118575080917) circle (2.pt);
\draw[color=black] (3.607331218042902,3.638108001139512) node {$6$};
\draw [fill=black] (2.8740890743334075,1.4477821103149966) circle (2.0pt);
\draw [fill=black] (0.4111475146938212,1.0341583369404102) circle (2.pt);
\draw[color=blue] (0.006924281623202161,1.7990080449054547) node {$\frac{1}{3}$};
\draw[color=blue] (1.18185503518343,1.3985723848435748) node {$\frac{1}{3}$};
\draw[color=blue] (.8621826344517337,0.39492159627059464) node {$\frac{1}{3}$};
\end{scriptsize}
\draw[color=black] (0.23512590253510582,0.9307523935967635) node {$\star$};
\draw node at (1.5,-1.5) {Step 1};

\end{tikzpicture}
\hspace*{.4cm}
%
\begin{tikzpicture}[line cap=round,line join=round,>=triangle 45,x=1.0cm,y=1.0cm, scale = .8]
\draw [line width=.8pt] (2.028040446976297,-0.5827345953420644)-- (1.1067874971874447,2.256228576456234);
\draw [line width=.8pt] (1.1067874971874447,2.256228576456234)-- (2.6636787078204596,1.3307124522628444);
\draw [line width=.8pt] (1.501610189954096,3.7227128638752225)-- (1.1067874971874447,2.256228576456234);
\draw [line width=.8pt] (1.501610189954096,3.7227128638752225)-- (2.6636787078204596,1.3307124522628444);
\draw [line width=.8pt] (1.501610189954096,3.7227128638752225)-- (3.430207224692275,3.4147118575080917);
\draw [line width=.8pt] (3.430207224692275,3.4147118575080917)-- (2.6636787078204596,1.3307124522628444);
\draw [line width=.8pt] (1.1067874971874447,2.256228576456234)-- (3.430207224692275,3.4147118575080917);
\draw [line width=.8pt] (-0.26569138719186647,1.2973734654515108)-- (2.6636787078204596,1.3307124522628444);
\draw [line width=.8pt] (-0.26569138719186647,1.2973734654515108)-- (2.028040446976297,-0.5827345953420644);
\draw [line width=.8pt] (-0.26569138719186647,1.2973734654515108)-- (1.1067874971874447,2.256228576456234);
\draw [line width=.8pt] (1.501610189954096,3.7227128638752225)-- (-0.26569138719186647,1.2973734654515108);
\begin{scriptsize}
\draw [fill=black] (2.028040446976297,-0.5827345953420644) circle (2.pt);
\draw[color=black] (1.990438285760426,-0.836549183549197) node {1};
\draw [fill=black] (2.6636787078204596,1.3307124522628444) circle (2.pt);
\draw[color=black] (2.8176858325095995,1.1751664414999283) node {$3$};
\draw [fill=black] (1.1067874971874447,2.256228576456234) circle (2.pt);
\draw[color=black] (1.501610189954096,2.3032312779760735) node {$4$};
\draw [fill=black] (1.501610189954096,3.7227128638752225) circle (2.pt);
\draw[color=black] (1.3135993838747382,4.051731774514098) node {$5$};
\draw[color=blue] (1.45188037063064813,3.0040755829014392) node {$\frac{5}{6}$};
\draw [fill=black] (3.430207224692275,3.4147118575080917) circle (2.pt);
\draw[color=black] (3.607331218042902,3.6005058399236405) node {$6$};
\draw [fill=black] (2.6636787078204587,1.3307124522628448) circle (2.0pt);
\draw [fill=black] (-0.26569138719186647,1.2973734654515108) circle (2.pt);
\draw[color=blue] (0.69575237742953215,1.1161745460594465) node {$\frac{1}{3}$};
\draw[color=blue] (0.7001664253326972,0.2561205156626589) node {$\frac{1}{3}$};
\draw[color=blue] (0.7025642641168256,1.7666102061213262) node {$\frac{5}{3}$};
\draw[color=blue] (0.5017469743898533,2.701459914086371) node {$\frac{5}{3}$};
\end{scriptsize}
\draw[color=black] (-0.5417129993505819,1.2819783281872215) node {$\star$};
\draw node at (1.7,-1.5) {Step 2};

\end{tikzpicture}

\begin{tikzpicture}[line cap=round,line join=round,>=triangle 45,x=1.0cm,y=1.0cm, scale = .8]
\draw [line width=.8pt] (0.8999756105001512,2.256228576456234)-- (2.6636787078204596,1.3307124522628444);
\draw [line width=.8pt] (1.407604786914417,3.8919225893466445)-- (0.8999756105001512,2.256228576456234);
\draw [line width=.8pt] (1.407604786914417,3.8919225893466445)-- (2.6636787078204596,1.3307124522628444);
\draw [line width=.8pt] (1.407604786914417,3.8919225893466445)-- (3.362917170139737,3.685110702659351);
\draw [line width=.8pt] (3.362917170139737,3.685110702659351)-- (2.6636787078204596,1.3307124522628444);
\draw [line width=.8pt] (0.8999756105001512,2.256228576456234)-- (3.362917170139737,3.685110702659351);
\draw [line width=.8pt] (-0.5665086769188388,1.4665831909229325)-- (2.6636787078204596,1.3307124522628444);
\draw [line width=.8pt] (1.407604786914417,3.8919225893466445)-- (-0.5665086769188388,1.4665831909229325);
\draw [line width=.8pt] (-0.5665086769188388,1.4665831909229325)-- (1.0127820941477659,0.5829324023499521);
\draw [line width=.8pt] (1.0127820941477659,0.5829324023499521)-- (0.8999756105001512,2.256228576456234);
\draw [line width=.8pt] (1.0127820941477659,0.5829324023499521)-- (1.5956155929937745,-0.770745401421422);
\begin{scriptsize}
\draw [fill=black] (1.5956155929937745,-0.770745401421422) circle (2.pt);
\draw[color=black] (1.6708199154255179,-0.9681567478047473) node {1};
\draw [fill=black] (2.6636787078204596,1.3307124522628444) circle (2.pt);
\draw[color=black] (2.836486913117535,1.3067740057554786) node {$3$};
\draw [fill=black] (0.8999756105001512,2.256228576456234) circle (2.pt);
\draw[color=black] (1.2947983032668025,2.3032312779760735) node {$4$};
\draw [fill=black] (1.407604786914417,3.8919225893466445) circle (2.pt);
\draw[color=black] (1.2759972226588667,4.258543661201392) node {$5$};
\draw [fill=black] (3.362917170139737,3.685110702659351) circle (2.pt);
\draw[color=black] (3.5697290568270303,3.901323129650612) node {$6$};
\draw [fill=black] (2.6636787078204587,1.3307124522628442) circle (2.0pt);
\draw [fill=black] (-0.5665086769188388,1.4665831909229325) circle (2.pt);
\draw[color=blue] (1.3259972226588667,3.2056831471569893) node {$\frac{5}{6}$};
\draw[color=blue] (0.21373616831049563,1.7050026418657759) node {$\frac{1}{3}$};
\draw[color=blue] (0.06828004080854091,2.6638577528704995) node {$\frac{5}{3}$};
\draw [fill=black] (1.0127820941477659,0.5829324023499521) circle (2.pt);
\draw[color=blue] (0.1573329264866883,0.8341583369404102) node {$\frac{5}{27}$};
\draw[color=blue] (1.1349891180993483,1.1469648205880247) node {$\frac{5}{9}$};
\draw[color=blue] (1.0973869568834766,-0.018702177103991885) node {$\frac{1}{9}$};
\end{scriptsize}
\draw[color=black] (1.2195939808350593,0.59233294265392) node {$\ast$};
\draw[color=black] (-0.8425302890775542,1.4511880536586433) node {$\star$};
\draw node at (1.5,-1.5) {Step 3};
\end{tikzpicture}
%
\hspace*{1.5cm}
\begin{tikzpicture}[line cap=round,line join=round,>=triangle 45,x=1.0cm,y=1.0cm, scale = .8]
\draw [line width=.8pt] (0.37354535347794965,2.049416689768941)-- (2.6636787078204596,1.3307124522628444);
\draw [line width=.8pt] (1.5768145123858388,3.5723042190117367)-- (0.37354535347794965,2.049416689768941);
\draw [line width=.8pt] (1.5768145123858388,3.5723042190117367)-- (2.6636787078204596,1.3307124522628444);
\draw [line width=.8pt] (1.5768145123858388,3.5723042190117367)-- (3.430207224692275,3.4147118575080917);
\draw [line width=.8pt] (3.430207224692275,3.4147118575080917)-- (2.6636787078204596,1.3307124522628444);
\draw [line width=.8pt] (0.37354535347794965,2.049416689768941)-- (3.430207224692275,3.4147118575080917);
\draw [line width=.8pt] (1.238395061442995,0.6957388859975666)-- (0.37354535347794965,2.049416689768941);
\draw [line width=.8pt] (1.238395061442995,0.6957388859975666)-- (0.9751799329318943,-0.7331432402055504);
\draw [line width=.8pt] (1.238395061442995,0.6957388859975666)-- (2.6636787078204596,1.3307124522628444);
\draw [line width=.8pt] (1.5768145123858388,3.5723042190117367)-- (1.238395061442995,0.6957388859975666);
\begin{scriptsize}
\draw [fill=black] (0.9751799329318943,-0.7331432402055504) circle (2.pt);
\draw[color=black] (0.9375777717160227,-0.986957828412683) node {1};
\draw [fill=black] (2.6636787078204596,1.3307124522628444) circle (2.pt);
\draw[color=black] (2.836486913117535,1.3067740057554786) node {$3$};
\draw [fill=black] (0.37354535347794965,2.049416689768941) circle (2.pt);
\draw[color=black] (-0.040078419896637255,2.115220471896716) node {$4$};
\draw [fill=black] (1.5768145123858388,3.5723042190117367) circle (2.pt);
\draw[color=black] (1.8588307215048756,3.9201242102585483) node {$5$};
\draw[color=blue] (0.8435723686763439,3.049279905333182) node {$\frac{5}{3}$};
\draw[color=blue] (2.366459897919141,2.3032312779760735) node {$\frac{5}{6}$};
\draw [fill=black] (3.430207224692275,3.4147118575080917) circle (2.pt);
\draw[color=black] (3.6261322986508375,3.4312961144522185) node {$6$};
\draw [fill=black] (2.663678707820459,1.3307124522628442) circle (2.0pt);
\draw [fill=black] (1.238395061442995,0.6957388859975666) circle (2.pt);
\draw[color=blue] (0.6085588610771467,1.3913788684911894) node {$\frac{5}{9}$};
\draw[color=blue] (0.9905750701961833,0.12571187079917288) node {$\frac{1}{9}$};
\draw[color=blue] (2.2690485515358158,0.889540145086672) node {$\frac{5}{9}$};
\draw[color=blue] (1.5046074884342564,1.1289810297070609) node {$\frac{25}{9}$};

\end{scriptsize}
\draw node at (1.5,-1.5) {Step 4};

\draw[color=black] (1.4640080287382244,0.59233294265392) node {$\ast$};

\end{tikzpicture}
    \caption{Demonstration of a 4-step algorithm exchanging a tetrahedron for a cut vertex using equivalent network transformations. Edges are assumed to have unit resistance unless labelled. }
    \label{fig:3tree}
\end{figure}
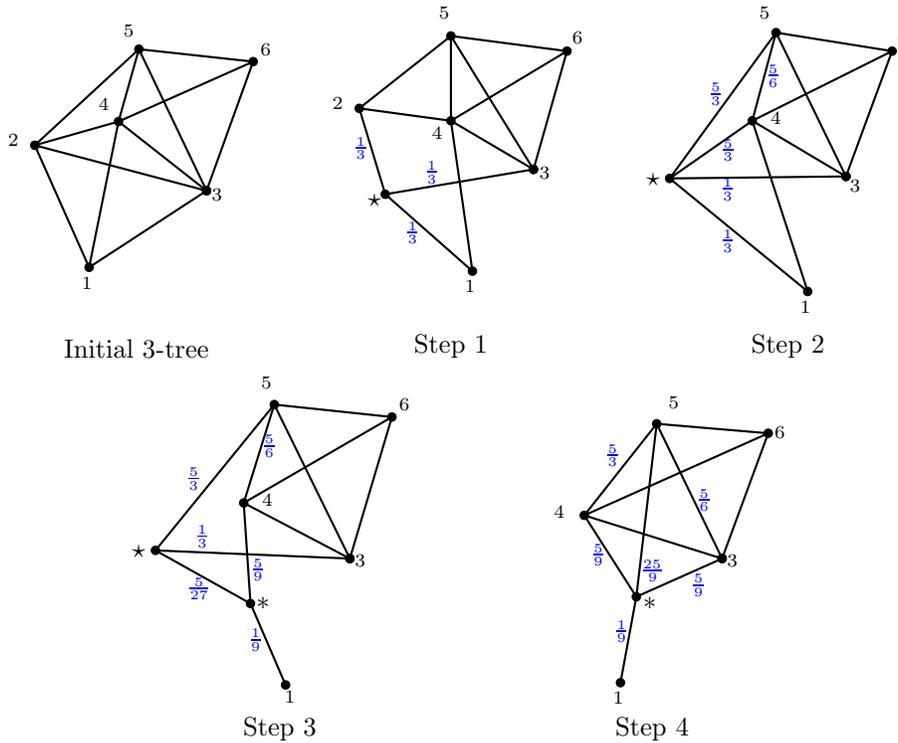

We next consider another family of graphs with considerable structure and symmetry. 

\begin{definition}

The triangular grid graph $T_n$ with $n$ rows and $\frac{(n+1)(n+2)}{2}$ vertices is defined recursively as follows: 
\begin{itemize}
\item  $T_1 = K_3$, the complete graph on 3 vertices. 
\item  $E(T_{n+1})$ is formed by adding to $E(T_n)$ the  edges 
\begin{multline*}
\Big(\frac{n(n+1)}{2}+k, \frac{(n+1)(n+2)}{2}+k\Big), \\
    \Big(\frac{n(n+1)}{2}+k, \frac{(n+1)(n+2)}{2}+k+1\Big), \quad \text{and} \\
    \Big(\frac{n(n+1)}{2}+k, \frac{(n+1)(n+2)}{2}+k\Big), 
    \end{multline*}
   for $k = 1, \ldots, n+1$. 
   \end{itemize}
   
 The  graph $T_4$ with 4 rows on 15 vertices is shown in Figure~\ref{fig:supertriangle}.   
\end{definition}

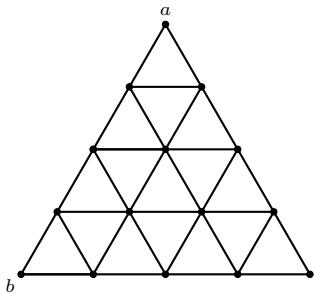
\begin{figure}[ht!]
\begin{center}
\begin{tikzpicture}[scale = .6, line cap=round,line join=round,>=triangle 45,x=1.0cm,y=1.0cm,scale = .8]
\draw [line width=.8pt,color=black] (-2.,10.)-- (-3.,11.732050807568879);
\draw [line width=.8pt,color=black] (-2.,10.)-- (-4.,10);
\draw [line width=.8pt,color=black] (-3.,11.732050807568879)-- (-4.,10.);
\draw [line width=.8pt,color=black] (-4.,10.)-- (-3.,8.267949192431121);
\draw [line width=.8pt,color=black] (-3.,8.267949192431121)-- (-2.,10.);
\draw [line width=.8pt,color=black] (-4.,10.)-- (-5.,8.267949192431123);
\draw [line width=.8pt,color=black] (-5.,8.267949192431123)-- (-3.,8.267949192431121);
\draw [line width=.8pt,color=black] (-3.,8.267949192431121)-- (-1.,8.267949192431121);
\draw [line width=.8pt,color=black] (-1.,8.267949192431121)-- (-2.,10.);
\draw [line width=.8pt,color=black] (-3.,8.267949192431121)-- (-2.,6.535898384862243);
\draw [line width=.8pt,color=black] (-3.,8.267949192431121)-- (-4.,6.535898384862244);
\draw [line width=.8pt,color=black] (-1.,8.267949192431121)-- (-2.,6.535898384862244);
\draw [line width=.8pt,color=black] (-4.,6.535898384862244)-- (-2.,6.535898384862243);
\draw [line width=.8pt,color=black] (-3.,8.267949192431121)-- (-5.,8.267949192431121);
\draw [line width=.8pt,color=black] (-5.,8.267949192431121)-- (-4.,6.535898384862244);
\draw [line width=.8pt,color=black] (-2.,6.535898384862243)-- (0.,6.535898384862242);
\draw [line width=.8pt,color=black] (0.,6.535898384862242)-- (-1.,8.267949192431121);
\draw [line width=.8pt,color=black] (-5.,8.267949192431123)-- (-6.,6.535898384862246);
\draw [line width=.8pt,color=black] (-6.,6.535898384862246)-- (-4.,6.535898384862244);
\draw [line width=.8pt,color=black] (-2.,6.535898384862243)-- (-1.,4.803847577293364);
\draw [line width=.8pt,color=black] (-1.,4.803847577293364)-- (0.,6.535898384862242);
\draw [line width=.8pt,color=black] (-4.,6.535898384862244)-- (-3.,4.8038475772933635);
\draw [line width=.8pt,color=black] (-3.,4.8038475772933635)-- (-2.,6.535898384862243);
\draw [line width=.8pt,color=black] (-6.,6.535898384862246)-- (-5.,4.803847577293364);
\draw [line width=.8pt,color=black] (-5.,4.803847577293364)-- (-4.,6.535898384862244);
\draw [line width=.8pt,color=black] (-6.,6.535898384862246)-- (-7.,4.80384757729337);
\draw [line width=.8pt,color=black] (1.,4.803847577293362)-- (0.,6.535898384862242);
\draw [line width=.8pt,color=black] (-7.,4.80384757729337)-- (-5.,4.803847577293364);
\draw [line width=.8pt,color=black] (-3.,4.80384757729337)-- (-5.,4.803847577293364);
\draw [line width=.8pt,color=black] (-3.,4.80384757729337)-- (-1.,4.803847577293364);
\draw [line width=.8pt,color=black] (-1.,4.803847577293364)-- (1.,4.803847577293362);
\draw [line width=.8pt,color=black] (-5.,4.803847577293364)-- (-7.,4.80384757729337);

\begin{scriptsize}
\draw [fill=black] (-4.,10.) circle (2.5pt);
\draw[color=black] (-3,12.1) node {$a$};
\draw [fill=black] (-2.,10.) circle (2.5pt);
\draw [fill=black] (-3.,11.732050807568879) circle (2.5pt);
\draw [fill=black] (-3.,8.267949192431121) circle (2.5pt);
\draw [fill=black] (-5.,8.267949192431123) circle (2.5pt);
\draw [fill=black] (-1.,8.267949192431121) circle (2.5pt);
\draw [fill=black] (-2.,6.535898384862243) circle (2.5pt);
\draw [fill=black] (-4.,6.535898384862244) circle (2.5pt);
\draw [fill=black] (-5.,8.267949192431121) circle (2.5pt);
\draw [fill=black] (0.,6.535898384862242) circle (2.5pt);
\draw [fill=black] (-6.,6.535898384862246) circle (2.5pt);
\draw [fill=black] (-1.,4.803847577293364) circle (2.5pt);
\draw [fill=black] (-3.,4.8038475772933635) circle (2.5pt);
\draw [fill=black] (-5.,4.803847577293364) circle (2.5pt);
\draw [fill=black] (-7.,4.80384757729337) circle (2.5pt);
\draw [fill=black] (1.,4.803847577293362) circle (2.5pt);
\draw[color=black] (-7.3,4.5) node {$b$};

\end{scriptsize}
\end{tikzpicture}
\end{center}
\caption{A triangular grid with $4$ rows.}
\label{fig:supertriangle}
\end{figure}

Note that this graph does not have bounded tree-width as the number of vertices goes to infinity.

Again, it seems natural to attempt to use equivalent network transformations simply or collapse this graph, in order to determine resistance distances (or, equivalently, the number of spanning two forests). Here we present an algorithm along these lines:

\begin{algorithm}\label{alg:supertri}
Given a triangular grid graph $T_n$ on $n$ rows,  we can perform equivalent network transformations in such a way as to reduce the number of rows separating the extremal nodes in the following way:
\begin{enumerate}
\item Perform a $\Delta$-Y transformation on each `upright' triangle, that is, each triangle with two vertices along the bottom edge and one vertex at the top. The transformed  graph after this step is shown in Panel (A) of Figure \ref{fig:supertri_alg}. The resistance on dashed edges is equal to 1/3. 
\item Perform series transformations on all exterior edges (except those incident to the three pendant vertices). The transformed  graph after this step is shown in Panel (B) of Figure \ref{fig:supertri_alg}. The resistance on the dotted edges is 2/3. 
\item Perform  Y-$\Delta$ transformations on all interior ``upright'' `Y's. The transformed  graph after this step is shown in Panel (C) of Figure \ref{fig:supertri_alg}. The resistance on each solid edge is equal to one. 
\end{enumerate}
\end{algorithm} 

Note that the transformed graph is simpler than the original, in that there are now fewer rows. However, the new graph is more complicated in the sense that the edge resistances are not the same for all edges.  The algorithm does provide a method whereby one can recover resistance distances in triangular grid graphs of arbitrary size, but it does not suggest a general closed formulae for these resistances.

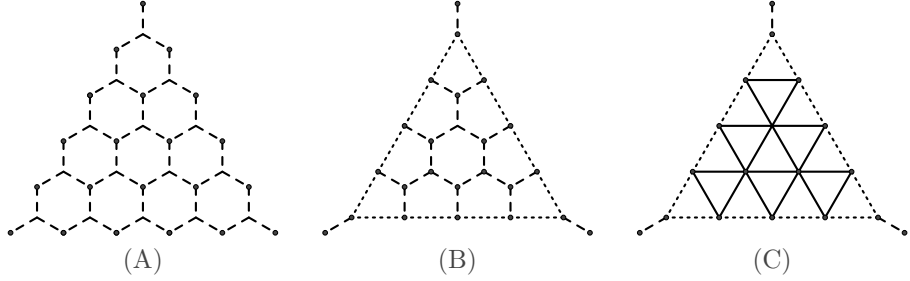
\begin{figure}[ht!]
    \centering
    \begin{center}
\begin{tabular}{ccc}
\definecolor{black}{rgb}{0.26666666666666666,0.26666666666666666,0.26666666666666666}
\definecolor{qqqqff}{rgb}{0.,0.,1.}
\definecolor{ttqqqq}{rgb}{0.2,0.,0.}

\begin{tikzpicture}[scale = .44, line cap=round,line join=round,>=triangle 45,x=1.0cm,y=1.0cm,scale=.8]
\draw [line width=.8pt,dashed] (-3.,11.732050807568879)-- (-3.,10.577350269189628);
\draw [line width=.8pt,dashed] (-3.,10.577350269189628)-- (-4.,10.);
\draw [line width=.8pt,dashed] (-3.,10.577350269189628)-- (-2.,10.);
\draw [line width=.8pt,dashed] (-4.,10.)-- (-4.,8.84529946162075);
\draw [line width=.8pt,dashed] (-4.,8.84529946162075)-- (-5.,8.267949192431123);
\draw [line width=.8pt,dashed] (-4.,8.84529946162075)-- (-3.,8.267949192431121);
\draw [line width=.8pt,dashed] (-2.,10.)-- (-2.,8.845299461620748);
\draw [line width=.8pt,dashed] (-2.,8.845299461620748)-- (-3.,8.267949192431121);
\draw [line width=.8pt,dashed] (-2.,8.845299461620748)-- (-1.,8.267949192431121);
\draw [line width=.8pt,dashed] (-5.,8.267949192431123)-- (-5.,7.113248654051867);
\draw [line width=.8pt,dashed] (-5.,7.113248654051867)-- (-6.,6.535898384862246);
\draw [line width=.8pt,dashed] (-5.,7.113248654051867)-- (-4.,6.535898384862244);
\draw [line width=.8pt,dashed] (-3.,8.267949192431121)-- (-3.,7.113248654051869);
\draw [line width=.8pt,dashed] (-3.,7.113248654051869)-- (-4.,6.535898384862244);
\draw [line width=.8pt,dashed] (-3.,7.113248654051869)-- (-2.,6.535898384862243);
\draw [line width=.8pt,dashed] (-1.,8.267949192431121)-- (-1.,7.113248654051868);
\draw [line width=.8pt,dashed] (-1.,7.113248654051868)-- (-2.,6.535898384862243);
\draw [line width=.8pt,dashed] (-1.,7.113248654051868)-- (0.,6.535898384862242);
\draw [line width=.8pt,dashed] (-6.,6.535898384862246)-- (-6.,5.381197846482991);
\draw [line width=.8pt,dashed] (-6.,5.381197846482991)-- (-7.,4.80384757729337);
\draw [line width=.8pt,dashed] (-6.,5.381197846482991)-- (-5.,4.803847577293364);
\draw [line width=.8pt,dashed] (-4.,6.535898384862244)-- (-4.,5.381197846482992);
\draw [line width=.8pt,dashed] (-4.,5.381197846482992)-- (-5.,4.803847577293364);
\draw [line width=.8pt,dashed] (-4.,5.381197846482992)-- (-3.,4.8038475772933635);
\draw [line width=.8pt,dashed] (-2.,6.535898384862243)-- (-2.,5.381197846482992);
\draw [line width=.8pt,dashed] (-2.,5.381197846482992)-- (-3.,4.8038475772933635);
\draw [line width=.8pt,dashed] (-2.,5.381197846482992)-- (-1.,4.803847577293364);
\draw [line width=.8pt,dashed] (0.,6.535898384862242)-- (0.,5.38119784648299);
\draw [line width=.8pt,dashed] (0.,5.38119784648299)-- (-1.,4.803847577293364);
\draw [line width=.8pt,dashed] (0.,5.38119784648299)-- (1.,4.803847577293362);
\draw [line width=.8pt,dashed] (1.,4.803847577293362)-- (1.,3.6491470389141094);
\draw [line width=.8pt,dashed] (1.,3.6491470389141094)-- (0.,3.071796769724484);
\draw [line width=.8pt,dashed] (1.,3.6491470389141094)-- (2.,3.0717967697244815);
\draw [line width=.8pt,dashed] (-1.,4.803847577293364)-- (-1.,3.6491470389141134);
\draw [line width=.8pt,dashed] (-1.,3.6491470389141134)-- (-2.,3.071796769724486);
\draw [line width=.8pt,dashed] (-1.,3.6491470389141134)-- (0.,3.071796769724484);
\draw [line width=.8pt,dashed] (-3.,4.8038475772933635)-- (-3.,3.6491470389141174);
\draw [line width=.8pt,dashed] (-3.,3.6491470389141174)-- (-4.,3.0717967697244837);
\draw [line width=.8pt,dashed] (-3.,3.6491470389141174)-- (-2.,3.071796769724486);
\draw [line width=.8pt,dashed] (-5.,4.803847577293364)-- (-5.,3.6491470389141125);
\draw [line width=.8pt,dashed] (-5.,3.6491470389141125)-- (-6.,3.071796769724488);
\draw [line width=.8pt,dashed] (-5.,3.6491470389141125)-- (-4.,3.0717967697244837);
\draw [line width=.8pt,dashed] (-7.,4.80384757729337)-- (-7.,3.6491470389141107);
\draw [line width=.8pt,dashed] (-7.,3.6491470389141107)-- (-8.,3.071796769724494);
\draw [line width=.8pt,dashed] (-7.,3.6491470389141107)-- (-6.,3.071796769724488);
\begin{scriptsize}
\draw [fill=black] (-4.,10.) circle (2.5pt);
\draw [fill=black] (-2.,10.) circle (2.5pt);
\draw [fill=black] (-3.,11.732050807568879) circle (2.5pt);
\draw [fill=black] (-3.,8.267949192431121) circle (2.5pt);
\draw [fill=black] (-5.,8.267949192431123) circle (2.5pt);
\draw [fill=black] (-1.,8.267949192431121) circle (2.5pt);
\draw [fill=black] (-2.,6.535898384862243) circle (2.5pt);
\draw [fill=black] (-4.,6.535898384862244) circle (2.5pt);
\draw [fill=black] (-5.,8.267949192431121) circle (2.5pt);
\draw [fill=black] (0.,6.535898384862242) circle (2.5pt);
\draw [fill=black] (-6.,6.535898384862246) circle (2.5pt);
\draw [fill=black] (-1.,4.803847577293364) circle (2.5pt);
\draw [fill=black] (-3.,4.8038475772933635) circle (2.5pt);
\draw [fill=black] (-5.,4.803847577293364) circle (2.5pt);
\draw [fill=black] (-7.,4.80384757729337) circle (2.5pt);
\draw [fill=black] (1.,4.803847577293362) circle (2.5pt);
\draw [fill=black] (0.,3.071796769724484) circle (2.5pt);
\draw [fill=black] (-2.,3.071796769724486) circle (2.5pt);
\draw [fill=black] (-4.,3.0717967697244837) circle (2.5pt);
\draw [fill=black] (-6.,3.071796769724488) circle (2.5pt);
\draw [fill=black] (-8.,3.071796769724494) circle (2.5pt);
\draw [fill=black] (2.,3.0717967697244815) circle (2.5pt);
\draw [fill=black] (-2.,3.0717967697244877) circle (2.5pt);
\draw [fill=black] (-4.,3.0717967697244863) circle (2.5pt);
\draw [fill=black] (-6.,3.0717967697244846) circle (2.5pt);
\end{scriptsize}
\draw[color=black] (-3,2) node {(A)};

\end{tikzpicture}&
\definecolor{black}{rgb}{0.26666666666666666,0.26666666666666666,0.26666666666666666}
\definecolor{qqqqff}{rgb}{0.,0.,1.}

\begin{tikzpicture}[scale = .44, line cap=round,line join=round,>=triangle 45,x=1.0cm,y=1.0cm,scale = .8]

\draw [line width=.8pt,dashed] (-3.,11.732050807568879)-- (-3.,10.577350269189628);
\draw [line width=.8pt,dashed] (1.,3.6491470389141094)-- (2.,3.0717967697244815);

\draw [line width=.8pt,dotted] (-3.,10.577350269189628)--  (-4.,8.84529946162075);

\draw [line width=.8pt,dotted] (-4.,8.84529946162075)--  (-5.,7.113248654051867);

\draw [line width=.8pt,dotted] (-5.,7.113248654051867)--  (-6.,5.381197846482991);

\draw [line width=.8pt,dotted] (-6.,5.381197846482991)-- (-7.,3.6491470389141107);

\draw [line width=.8pt,dotted] (-3.,10.577350269189628)-- (-2.,8.845299461620748);

\draw [line width=.8pt,dotted] (-2.,8.845299461620748)--  (-1.,7.113248654051868);

\draw [line width=.8pt,dotted] (-1.,7.113248654051868)-- (0.,5.38119784648299);

\draw [line width=.8pt,dotted] (0.,5.38119784648299)-- (1.,3.6491470389141094);

\draw [line width=.8pt,dotted] (1.,3.6491470389141094)-- (-1.,3.6491470389141134);

\draw [line width=.8pt,dotted] (-3.,3.6491470389141174)-- (-1.,3.6491470389141134);

\draw [line width=.8pt,dotted] (-3.,3.6491470389141174)--  (-5.,3.6491470389141125);

\draw [line width=.8pt,dotted] (-5.,3.6491470389141125)-- (-7.,3.6491470389141107);

\draw [line width=.8pt,dashed] (-7.,3.6491470389141107)-- (-8.,3.071796769724494);

\draw [line width=.8pt,dashed] (-2.,6.535898384862243)-- (-2.,5.381197846482992);
\draw [line width=.8pt,dashed] (-4.,8.84529946162075)-- (-3.,8.267949192431121);
\draw [line width=.8pt,dashed] (-2.,8.845299461620748)-- (-3.,8.267949192431121);
\draw [line width=.8pt,dashed] (-5.,7.113248654051867)-- (-4.,6.535898384862244);
\draw [line width=.8pt,dashed] (-3.,8.267949192431121)-- (-3.,7.113248654051869);
\draw [line width=.8pt,dashed] (-3.,7.113248654051869)-- (-4.,6.535898384862244);
\draw [line width=.8pt,dashed] (-3.,7.113248654051869)-- (-2.,6.535898384862243);
\draw [line width=.8pt,dashed] (-1.,7.113248654051868)-- (-2.,6.535898384862243);
\draw [line width=.8pt,dashed] (-6.,5.381197846482991)-- (-5.,4.803847577293364);
\draw [line width=.8pt,dashed] (-4.,6.535898384862244)-- (-4.,5.381197846482992);
\draw [line width=.8pt,dashed] (-4.,5.381197846482992)-- (-5.,4.803847577293364);
\draw [line width=.8pt,dashed] (-4.,5.381197846482992)-- (-3.,4.8038475772933635);
\draw [line width=.8pt,dashed] (-2.,5.381197846482992)-- (-3.,4.8038475772933635);
\draw [line width=.8pt,dashed] (-2.,5.381197846482992)-- (-1.,4.803847577293364);
\draw [line width=.8pt,dashed] (0.,5.38119784648299)-- (-1.,4.803847577293364);
\draw [line width=.8pt,dashed] (-1.,4.803847577293364)-- (-1.,3.6491470389141134);
\draw [line width=.8pt,dashed] (-3.,4.8038475772933635)-- (-3.,3.6491470389141174);
\draw [line width=.8pt,dashed] (-5.,4.803847577293364)-- (-5.,3.6491470389141125);
\begin{scriptsize}
\draw [fill=black] (-3.,11.732050807568879) circle (2.5pt);
\draw [fill=black] (-3.,8.267949192431121) circle (2.5pt);
\draw [fill=black] (-2.,6.535898384862243) circle (2.5pt);
\draw [fill=black] (-4.,6.535898384862244) circle (2.5pt);
\draw [fill=black] (-1.,4.803847577293364) circle (2.5pt);
\draw [fill=black] (-3.,4.8038475772933635) circle (2.5pt);
\draw [fill=black] (-5.,4.803847577293364) circle (2.5pt);
\draw [fill=black] (-8.,3.071796769724494) circle (2.5pt);
\draw [fill=black] (2.,3.0717967697244815) circle (2.5pt);
\draw [fill=black] (-3.,10.577350269189628) circle (2.5pt);
\draw [fill=black] (-2.,5.381197846482992) circle (2.5pt);
\draw [fill=black] (-1.,7.113248654051868) circle (2.5pt);
\draw [fill=black] (-2.,8.845299461620748) circle (2.5pt);
\draw [fill=black] (-4.,8.84529946162075) circle (2.5pt);
\draw [fill=black] (-5.,7.113248654051867) circle (2.5pt);
\draw [fill=black] (-4.,5.381197846482992) circle (2.5pt);
\draw [fill=black] (-1.,3.6491470389141134) circle (2.5pt);
\draw [fill=black] (0.,5.38119784648299) circle (2.5pt);
\draw [fill=black] (1.,3.6491470389141094) circle (2.5pt);
\draw [fill=black] (-6.,5.381197846482991) circle (2.5pt);
\draw [fill=black] (-5.,3.6491470389141125) circle (2.5pt);
\draw [fill=black] (-7.,3.6491470389141107) circle (2.5pt);
\draw [fill=black] (-3.,3.6491470389141174) circle (2.5pt);

\end{scriptsize}
\draw[color=black] (-3,2) node {(B)};

\end{tikzpicture}
& 
\definecolor{black}{rgb}{0.26666666666666666,0.26666666666666666,0.26666666666666666}
\definecolor{qqqqff}{rgb}{0.,0.,1.}

\begin{tikzpicture}[scale = .44, line cap=round,line join=round,>=triangle 45,x=1.0cm,y=1.0cm,scale=.8]

\draw [line width=.8pt,dashed] (-3.,11.732050807568879)-- (-3.,10.577350269189628);
\draw [line width=.8pt,dashed] (1.,3.6491470389141094)-- (2.,3.0717967697244815);

\draw [line width=.8pt,dotted] (-3.,10.577350269189628)--  (-4.,8.84529946162075);

\draw [line width=.8pt,dotted] (-4.,8.84529946162075)--  (-5.,7.113248654051867);

\draw [line width=.8pt,dotted] (-5.,7.113248654051867)--  (-6.,5.381197846482991);

\draw [line width=.8pt,dotted] (-6.,5.381197846482991)-- (-7.,3.6491470389141107);

\draw [line width=.8pt,dotted] (-3.,10.577350269189628)-- (-2.,8.845299461620748);

\draw [line width=.8pt,dotted] (-2.,8.845299461620748)--  (-1.,7.113248654051868);

\draw [line width=.8pt,dotted] (-1.,7.113248654051868)-- (0.,5.38119784648299);

\draw [line width=.8pt,dotted] (0.,5.38119784648299)-- (1.,3.6491470389141094);
\draw [line width=.8pt,dotted] (1.,3.6491470389141094)-- (-1.,3.6491470389141134);
\draw [line width=.8pt,dotted] (-3.,3.6491470389141174)-- (-1.,3.6491470389141134);
\draw [line width=.8pt,dotted] (-3.,3.6491470389141174)--  (-5.,3.6491470389141125);
\draw [line width=.8pt,dotted] (-5.,3.6491470389141125)-- (-7.,3.6491470389141107);
\draw [line width=.8pt,dashed] (-7.,3.6491470389141107)-- (-8.,3.071796769724494);
\draw [line width=.8pt] (-4.,8.84529946162075)--(-2.,8.845299461620748)-- (-3.,7.113248654051869)--(-4.,8.84529946162075);
\draw [line width=.8pt] (-3.,7.113248654051869)-- (-5.,7.113248654051867)--  (-4.,5.381197846482992)--(-3.,7.113248654051869);
\draw [line width=.8pt] (-3.,7.113248654051869)--  (-2.,5.381197846482992)--(-1.,7.113248654051868)--(-3.,7.113248654051869);
\draw [line width=.8pt] (-6.,5.381197846482991)--(-4.,5.381197846482992)-- (-5.,3.6491470389141125)--(-6.,5.381197846482991);
\draw [line width=.8pt] (-4.,5.381197846482992)--  (-2.,5.381197846482992)-- (-3.,3.6491470389141174)--(-4.,5.381197846482992);
\draw [line width=.8pt] (-2.,5.381197846482992)-- (0.,5.38119784648299)-- (-1.,3.6491470389141134)--(-2.,5.381197846482992);
\begin{scriptsize}
\draw [fill=black] (-3.,11.732050807568879) circle (2.5pt);
\draw [fill=black] (-8.,3.071796769724494) circle (2.5pt);
\draw [fill=black] (2.,3.0717967697244815) circle (2.5pt);
\draw [fill=black] (-3.,10.577350269189628) circle (2.5pt);
\draw [fill=black] (-2.,5.381197846482992) circle (2.5pt);
\draw [fill=black] (-1.,7.113248654051868) circle (2.5pt);
\draw [fill=black] (-2.,8.845299461620748) circle (2.5pt);
\draw [fill=black] (-4.,8.84529946162075) circle (2.5pt);
\draw [fill=black] (-5.,7.113248654051867) circle (2.5pt);
\draw [fill=black] (-4.,5.381197846482992) circle (2.5pt);
\draw [fill=black] (-1.,3.6491470389141134) circle (2.5pt);
\draw [fill=black] (0.,5.38119784648299) circle (2.5pt);
\draw [fill=black] (1.,3.6491470389141094) circle (2.5pt);
\draw [fill=black] (-6.,5.381197846482991) circle (2.5pt);
\draw [fill=black] (-5.,3.6491470389141125) circle (2.5pt);
\draw [fill=black] (-7.,3.6491470389141107) circle (2.5pt);
\draw [fill=black] (-3.,3.6491470389141174) circle (2.5pt);
\end{scriptsize}
\draw[color=black] (-3,2) node {(C)};

\end{tikzpicture}
\end{tabular}
\end{center}
    \caption{Panel (A) shows the graph after Step 1 in Algorithm \ref{alg:supertri}, Panel (B) after Step 2, and Panel (C) after Step 3. Dashed edges have resistance  $\frac13$, dotted edges $\frac23$ and solid edges $1$.}
    \label{fig:supertri_alg}
\end{figure}

 This inspires the following open questions:
\begin{question}
In what cases can an algorithm using equivalent network transformations (found in Section~\ref{sec:networktransformations}) together with Theorem \ref{thm:cutvertex} be used to produce resistance distance between some or all pairs of nodes in the original network?
\end{question}

As mentioned above, it is possible compute the maximal effective resistance between the extremal nodes in both straight linear 3-trees and triangular grid graphs for small $n$, but no closed formula is known in general in either case. However, empirical evidence has allowed us to make the following conjectures:
\begin{conjecture}
	Let $G$ be the straight linear K-tree, $k \geq 1$, with $n$ vertices and $H$ be the straight linear K-tree with $n+1$ vertices. 
	Then 
	\[\lim_{n\rightarrow \infty} r_{H} (1, n+1) - r_G(1,n) = \frac{6}{k(k+1)(2k+1)}.\]
	\end{conjecture}
	
\begin{conjecture}
Let $G$ be the block tower graph $C_4 \square P_n$ (that is the Cartesian product of the 4--cycle and the path of length $n$)  with $4n$ vertices and $H$ be the block tower graph $C_4 \square P_{n+1}$ with $4n+4$ vertices.  Then 
	\[\lim_{n\rightarrow \infty} r_{H} (1, 4n+3) - r_G(1,4n-1) = \frac{1}{4}.\]

\end{conjecture}
\begin{conjecture}
Let $G$ be the $n\times n$ grid graph $P_n \square P_n$ with $n^2$ vertices and $H$ be the $(n+1)\times (n+1)$ grid graph $P_{n+1}\square P_{n+1}$ with $(n+1)^2$ vertices.  Then 
\[\lim_{n\rightarrow \infty} \exp(r_{H}(a,b)) -\exp(r_G(a,b)) = C > 0.\]
Moreover $\lim_{n\rightarrow \infty} r_n(a,b) = \infty.$
\end{conjecture}

 \begin{conjecture}
Let $T_n$ be the triangular grid graph shown in Figure~\ref{fig:supertriangle} with $n$ rows and $m=n^2$ cells.  Moreover let $a$ and $b$ be distinct vertices with degree 2 and let $r_n(a,b)$ be the resistance distance between $a$ and $b$ in $T_n$.  Then
\[\lim_{n\rightarrow \infty} \exp(r_{n+1}(a,b)) -\exp(r_n(a,b)) = C > 0.\]
Moreover $\lim_{n\rightarrow \infty} r_n(a,b) = \infty.$
\end{conjecture}

It is relatively easy to give a lower bound on the resistance between nodes $a$ and $b$ in the triangular grid graph.  Suppose each horizontal edge is shorted, that is to say we add a wire with zero resistance parallel to each horizontal edge in the triangular grid graph.  Adding such an edge decreases the resistance between $a$ and $b$ .  This has the effect of removing all the horizontal edges and merging all nodes along each row. The resulting circuit is a path starting at node $a=1$ and ending at node $b=n$, where the number of paths between nodes $i$ and $j$ is $2i$.  Applying the parallel and series rule, the resulting resistance between nodes $a$ and $b$ is \[\textstyle\frac{1}{2}\left(1+\frac{1}{2}+ \frac{1}{3}+\cdots + \frac{1}{n}\right).\]
Thus we have a partial sum of harmonic series and it is easy to see as $n$ goes to infinity the sum goes to infinity. (Thanks to H.Tracy Hall for this argument.)

Recall from Theorems \ref{thm:cutvertex} and \ref{thm:2sep2switch}, that if $G$ can be separated by one or two vertices, then the resistance distances and number of separating spanning 2-forests can be recovered from those in the two component graphs of the separation. It is natural then to ask if similar results can be stated for separations of larger size. 

\begin{question}
Suppose that $G$ is a graph with subgraphs $G_1$ and $G_2$ such that 
\begin{itemize}
    \item $V(G) = V(G_1) \cup V(G_2)$,
    \item $|V(G_1) \cap V(G_2)| = n$ for $n\geq 3$,
    \item $E(G) = E(G_1) \cup E(G_2)$, and 
    \item $E(G_1) \cap E(G_2) = \emptyset$.
\end{itemize}
What can be said about how  resistance distances in $G$ are related to resistance distances in $G_1$ and $G_2$?
\end{question}
It is worth noting that the inspiration for Theorem~\ref{thm:2sep2switch} came from the study of 2-separators for determining the maximum co-rank of a graph~\cite{van2008maximum}.  To our knowledge, this question has not been answered for determining the maximum co-rank of a graph with 3-separators.

\section*{Acknowledgments}
We benefited from many very helpful conversations with colleagues at BYU, including Wayne Barrett, Tracy Hall, Mark Kempton, Ben Webb, and others.  This paper is a direct result of research done at the Mathematical Sciences Research Institute as part of their Summer Research in Mathematics (SRiM) program.  

\bibliography{references}{}

\def\cprime{$'$}
\begin{thebibliography}{10}

\bibitem{rdmatrix}
D.~Babić, D.~J. Klein, I.~Lukovits, S.~Nikolić, and N.~Trinajstić.
\newblock Resistance-distance matrix: A computational algorithm and its
  application.
\newblock {\em International Journal of Quantum Chemistry}, 90(1):166--176,
  2002.

\bibitem{bapatdvi}
R.~B. Bapat.
\newblock Resistance distance in graphs.
\newblock {\em Math. Student}, 68(1-4):87--98, 1999.

\bibitem{BapatWheels}
R.~B. Bapat and Somit Gupta.
\newblock Resistance distance in wheels and fans.
\newblock {\em Indian. J. Pure App. Math.}, 41(1):1--13, Feb 2010.

\bibitem{Bapatbook}
R.B. Bapat.
\newblock {\em Graphs and Matrices}.
\newblock Universitext. Springer London, 2010.

\bibitem{BAPAT20111479}
R.B. Bapat and Sivaramakrishnan Sivasubramanian.
\newblock Identities for minors of the laplacian, resistance and distance
  matrices.
\newblock {\em Linear Algebra and its Applications}, 435(6):1479--1489, 2011.

\bibitem{Barooah06grapheffective}
Prabir Barooah and Joao~P. Hespanha.
\newblock Graph effective resistances and distributed control: Spectral
  properties and applications.
\newblock In {\em In Proc. of the 45th IEEE Conference on Decision and
  Control}, pages 3479--3485, 2006.

\bibitem{bef}
Wayne Barrett, Emily.~J. Evans, and Amanda~E. Francis.
\newblock Resistance distance in straight linear 2-trees.
\newblock {\em Discrete Appl. Math.}, 258:13--34, 2019.

\bibitem{swim2019}
Wayne Barrett, Emily~J. Evans, and Amanda~E. Francis.
\newblock Resistance distance and spanning 2-forest matrices of linear 2-trees.
\newblock {\em Linear Algebra Appl.}, 606:41--67, 2020.

\bibitem{bent2tree}
Wayne Barrett, Emily~J. Evans, Amanda~E. Francis, Mark Kempton, and John
  Sinkovic.
\newblock Spanning 2-forests and resistance distance in 2-connected graphs.
\newblock {\em Discrete Applied Mathematics}, 284:341 -- 352, 2020.

\bibitem{mgt}
B\'{e}la Bollob\'{a}s.
\newblock {\em Modern graph theory}, volume 184 of {\em Graduate Texts in
  Mathematics}.
\newblock Springer-Verlag, New York, 1998.

\bibitem{carmona2014effective}
Angeles Carmona, Andres~M Encinas, and Margarida Mitjana.
\newblock Effective resistances for ladder-like chains.
\newblock {\em International journal of quantum chemistry}, 114(24):1670--1677,
  2014.

\bibitem{chen2}
Haiyan Chen.
\newblock Random walks and the effective resistance sum rules.
\newblock {\em Discrete Applied Mathematics}, 158(15):1691--1700, 2010.

\bibitem{chenzhang}
Haiyan Chen and Fuji Zhang.
\newblock Resistance distance and the normalized {L}aplacian spectrum.
\newblock {\em Discrete Appl. Math.}, 155(5):654--661, 2007.

\bibitem{chen2008resistance}
Haiyan Chen and Fuji Zhang.
\newblock Resistance distance local rules.
\newblock {\em Journal of mathematical chemistry}, 44(2):405--417, 2008.

\bibitem{Cinkir}
Zubeyir Cinkir.
\newblock Effective resistances and {K}irchhoff index of ladder graphs.
\newblock {\em J. Math. Chem.}, 54(4):955--966, 2016.

\bibitem{DAFONSECA2007283}
C.M. {da Fonseca}.
\newblock On the eigenvalues of some tridiagonal matrices.
\newblock {\em Journal of Computational and Applied Mathematics},
  200(1):283--286, 2007.

\bibitem{devriendt2020effective}
Karel Devriendt.
\newblock Effective resistance is more than distance: Laplacians, simplices and
  the schur complement, 2020.

\bibitem{doylesnell}
Peter~G. Doyle and J.~Laurie Snell.
\newblock {\em Random walks and electric networks}, volume~22 of {\em Carus
  Mathematical Monographs}.
\newblock Mathematical Association of America, Washington, DC, 1984.

\bibitem{MarkK}
Nolan Faught, Mark Kempton, and Adam Knudson.
\newblock Resistance distance, {K}irchhoff index, and {K}emeny's constant in
  flower graphs, 2020.

\bibitem{fiedlerLapl}
Miroslav Fiedler.
\newblock Some characterizations of symmetric inverse {$M$}-matrices.
\newblock In {\em Proceedings of the {S}ixth {C}onference of the
  {I}nternational {L}inear {A}lgebra {S}ociety ({C}hemnitz, 1996)}, volume
  275/276, pages 179--187, 1998.

\bibitem{Ghosh}
Arpita Ghosh, Stephen Boyd, and Amin Saberi.
\newblock Minimizing effective resistance of a graph.
\newblock {\em SIAM Rev.}, 50(1):37--66, 2008.

\bibitem{KleinRandic}
D.~J. Klein and M.~Randi{\'{c}}.
\newblock Resistance distance.
\newblock {\em J. Math. Chem.}, 12(1):81--95, Dec 1993.

\bibitem{klein1997graph}
DJ~Klein.
\newblock Graph geometry, graph metrics and {W}iener.
\newblock {\em MATCH Commun. Math. Comput. Chem}, 35(7), 1997.

\bibitem{klein2002resistance}
Douglas~J Klein.
\newblock Resistance-distance sum rules.
\newblock {\em Croatica chemica acta}, 75(2):633--649, 2002.

\bibitem{Pachev}
Benjamin Pachev and Benjamin Webb.
\newblock {Fast link prediction for large networks using spectral embedding}.
\newblock {\em Journal of Complex Networks}, 6(1):79--94, 07 2017.

\bibitem{kem1}
Jos{e}'~Luis Palacios.
\newblock On the {K}irchhoff index of regular graphs.
\newblock {\em International Journal of Quantum Chemistry}, 110(7):1307--1309,
  2010.

\bibitem{peng2017kirchhoff}
YJ~Peng and SC~Li.
\newblock On the {K}irchhoff index and the number of spanning trees of linear
  phenylenes.
\newblock {\em MATCH Commun. Math. Comput. Chem}, 77(3):765--780, 2017.

\bibitem{OEIS}
Neil J.~A. Sloane and The OEIS~Foundation Inc.
\newblock The on-line encyclopedia of integer sequences, 2021.

\bibitem{SpielSparse}
Daniel~A. Spielman and Nikhil Srivastava.
\newblock Graph sparsification by effective resistances.
\newblock {\em SIAM J. Comput.}, 40(6):1913--1926, 2011.

\bibitem{oldbook}
William Stevenson.
\newblock {\em Elements of Power System Analysis}.
\newblock McGraw Hill, New York, 3 edition, 1975.

\bibitem{sweet}
Roland~A. Sweet.
\newblock A recursive relation for the determinant of a pentadiagonal matrix.
\newblock {\em Commun. ACM}, 12(6):330–332, June 1969.

\bibitem{van2008maximum}
Hein van~der Holst.
\newblock The maximum corank of graphs with a 2-separation.
\newblock {\em Linear Algebra and its Applications}, 428(7):1587--1600, 2008.

\bibitem{vos2016methods}
Vaya Sapobi~Samui Vos.
\newblock {\em Methods for determining the effective resistance}.
\newblock PhD thesis, Masters thesis, 20 December, 2016.

\bibitem{WANG201512}
Chaojie Wang, Hongyi Li, and Di~Zhao.
\newblock An explicit formula for the inverse of a pentadiagonal toeplitz
  matrix.
\newblock {\em Journal of Computational and Applied Mathematics}, 278:12--18,
  2015.

\bibitem{wang2010kirchhoff}
Yan Wang and Wenwen Zhang.
\newblock Kirchhoff index of linear pentagonal chains.
\newblock {\em International Journal of Quantum Chemistry}, 110(9):1594--1604,
  2010.

\bibitem{YangKlein}
Yujun Yang and Douglas~J. Klein.
\newblock A recursion formula for resistance distances and its applications.
\newblock {\em Discrete Appl. Math.}, 161(16-17):2702--2715, November 2013.

\bibitem{yang2014comparison}
Yujun Yang and Douglas~J Klein.
\newblock Comparison theorems on resistance distances and {K}irchhoff indices
  of s, t-isomers.
\newblock {\em Discrete Applied Mathematics}, 175:87--93, 2014.

\bibitem{yang2008kirchhoff}
Yujun Yang and Heping Zhang.
\newblock Kirchhoff index of linear hexagonal chains.
\newblock {\em International journal of quantum chemistry}, 108(3):503--512,
  2008.

\bibitem{Zhao2008OnTI}
Xi-Le Zhao and T.~Huang.
\newblock On the inverse of a general pentadiagonal matrix.
\newblock {\em Appl. Math. Comput.}, 202:639--646, 2008.

\end{thebibliography}
\bibliographystyle{plain}

\end{document}